\documentclass{amsart}

\usepackage[colorlinks=true,urlcolor=blue,linkcolor=blue,citecolor=black]{hyperref}
\usepackage{amssymb,amsmath, amscd, mathrsfs, yfonts, bbm, graphics, dsfont, footmisc }
\usepackage{xcolor}
\usepackage{tikz}
\usepackage{caption}

\title[] {Asymptotic free independence and entry permutations for Gaussian random matrices.\\ Part II: infinitesimal freeness}

\author[M. Popa]{Mihai Popa}
\address[M. Popa]{Department of Mathematics\\ University of Texas at San Antonio\\ One UTSA Circle San Antonio\\ Texas 78249, USA
	\\ and “Simon Stoilow” Institute of Mathematics of the Romanian Academy\\ P.O. Box 1-764\\ 014700 Bucharest, Romania}
\email{mihai.popa@utsa.edu}

\author[K. Szpojankowski]{Kamil Szpojankowski}
\address[K. Szpojankowski]{
	Wydzia\l{} Matematyki i Nauk Informacyjnych\\
	Politechnika Warszawska\\
	ul. Koszykowa 75\\
	00-662 Warszawa, Poland.}
\email{kamil.szpojankowski@pw.edu.pl}
\thanks{KSz: This research was funded in part by National Science Centre, Poland WEAVE-UNISONO grant BOOMER 2022/04/Y/ST1/00008.
\\ For the purpose of Open Access, the authors have applied a CC-BY public copyright licence to any Author Accepted Manuscript (AAM) version arising from this submission.}

\author[P.-L. Tseng]{Pei-Lun Tseng}
\address[P. -L. Tseng]{Department of Mathematics\\ New York University Abu Dhabi\\
Saadiyat Marina District, Abu Dhabi, United Arab Emirates
}
\email{pt2270@nyu.edu}

\newtheorem{claim}{}[section]
\newtheorem{defn}[claim]{Definition}
\newtheorem{thm}[claim]{Theorem}
\newtheorem{lemma}[claim]{Lemma}
\newtheorem{remark}[claim]{Remark}

\newcommand{\Tr}{\textrm{Tr}}

\newcommand{\tr}{\mathrm{tr}}

\newcommand{\id}{\text{Id}}

\input xy
\xyoption{all}

\usepackage{graphicx}

\begin{document}

\begin{abstract}
We study asymptotic infinitesimal distributions of Gaussian Unitary Ensembles with permuted entries. We show that for a uniformly random permutation the asymptotically permuted GUE matrix has a null infinitesimal distribution. Moreover, we show that asymptotically different permutations of the same GUE matrix are infinitesimally free. Besides this we study a particular example of entry permutation - the transpose, and we show that while a GUE matrix is asymptotically free from its transpose it is not infinitesimally free from it.
\end{abstract}

\maketitle


\section{Introduction}

Free probability introduced by D. Voiculescu \cite{voi86,voi92} is a non--commutative analogue of classical probability theory, where classical random variables are replaced by non--commutative operators. This theory was introduced in connection to some fundamental problems from the theory of operator algebras (such as the Free Group Factors isomorphism problem), however very quickly (see \cite{voi91}) it was noticed that the novel theory has deep connections with the random matrix theory. Since then, the connections between free probability and random matrices have become a very active field with many advances in recent years (see \cite{MMPS,BBC,BMS}). A rough explanation of the phenomenon behind is as follows -- big unitarily invariant and independent random matrices are \emph{asymptotically free}. By this we mean that for two sequences of random matrices the quantity $\tfrac{1}{N} E\left(\Tr(P(A_N,B_N))\right)$, for a non-commutative polynomial $P$ converges, as matrix size $N$ goes to infinity, to $\varphi\left(P(a,b)\right)$ where $a,b$ are free random variables with respect to $\varphi$, and $a,b$ have distributions equal to the weak limits of the expected empirical eigenvalue distributions of $A_N$ and $B_N$ respectively.

In recent years some attention was given to $1/N$ correction in the convergence above (see e.g. \cite{shly18}), i.e. to look not only at $\varphi$ but also consider another functional $\varphi'$ which satisfies
\begin{align}\label{eqn:introPhi'}
    \frac{1}{N} E\left(\Tr(P(A_N,B_N))\right)=\varphi(P(a,b))+\frac{1}{N} \varphi'(P(a,b))+o\left(\frac{1}{N}\right) \mbox{ as } N\to\infty.
\end{align}
The special interest from non-commutative probability point of view is the case when $a,b$ are infinitesimally free with respect to the pair of functionals $\left(\varphi,\varphi'\right)$ (we state the precise definition of infinitesimal freeness in the next section). This notion, under in the framework free probability of type B, was introduced in \cite{BGN} and the infinitesimal freeness interpretation was found in \cite{BS12}, see also \cite{FN10} for combinatorial developments related to infinitesimal freeness. As mentioned above asymptotic freeness connects directly to random matrix theory and several classes of random matrices have been proved to be asymptotically infinitesimally free \cite{mingo19,DF19,mingo21} and many related properties have been discovered \cite{FN10,Ts22,Ts19,CE19,F12,HS11}.  

Another recent development in random matrix theory is that freeness emerges also when one looks at different entry permutations of a given random matrix. In order to explain this phenomenon let us fix some notation. For any $N\geq 1$ let $A_N=(a_{i,j})$ be an $N\times N$ matrix and $\sigma_N:[N]^2\to [N]^2$ be a bijection; that is, $\sigma_N$ is a permutation on $[N]^2:=\{1,\dots,N\}^2$. By $A_N^{\sigma_N}$ we denote the permuted matrix, that is we have  $[A_N^{\sigma_N}]_{i,j}=a_{\sigma(i,j)}$. Among permutations there are many interesting mappings such as partial transposes, which are of interest in quantum information theory \cite{AGSSW,BTNI,HMHP}, and the mixed map in quantum physics \cite{DTDS,MALT}. The connection between matrix permutation and free probability were also explored in
\cite{MP16,MP19,MP20, mvp_gaussian,mpsz1, mpsz2}. In \cite{mvp_gaussian} the author showed that for a given sequence of Gaussian random matrices $(G_N)_N$, and two sequences of permutations of these matrices $\left(G^{\sigma_N}_N\right)_N$ and $\left(G^{\tau_N}_N\right)_N$ the permuted matrices are asymptotically (as $N\to \infty$) circular and asymptotically  free, whenever pairs of permutation sequences $(\sigma_N)_N$ and $(\tau_N)_N$ satisfy certain conditions. Moreover, following these developments in \cite{mpsz1}, the authors proved that such conditions occurs with probability one, which means that $G^{\sigma_N}_N$ and $G^{\tau_N}_N$ are asymptotically circular and asymptotically free for almost all pairs of independent permutation sequences $(\sigma_N)_N$ and $(\tau_N)_N$.

In the present paper we show that the framework described above gives not only the asymptotic freeness but also asymptotic infinitesimal freeness. More precisely we show that:
\begin{itemize}
    \item Asymptotically the infinitesimal distribution of a randomly permuted (with uniformly chosen permutation of entries) growing GUE matrix is zero.
    \item Independent permutations of a sequence of growing GUE matrices are asymptotically infinitesimally free.
    \item a
    GUE matrix and its transpose are not asymptotically infinitesimally free, even though they are
    asymptotically free \cite{MP16}, but we can explicitly compute their joint infinitesimal distribution.
\end{itemize}
Moreover we show that the phenomenon described above does not hold for any sequence of matrix permutations, namely we show that the sequence of GUE matrices is not asymptotically infinitesimally free from its transposes, although asymptotic freeness takes place as it was shown in \cite{MP16}.

The paper is organized as follows. We review the framework and properties of infinitesimal freeness in Section 2. In addition, some notation and a basic lemma on permuted Gaussian matrices are also included in Section 2. The infinitesimal distribution of the generic permuted Gaussian matrix is considered in Section 3. 

In Section 4, we consider pairs of independent sequences of random permutations $(\sigma_N)_N$ and $(\tau_N)_N$, such that that $\sigma_N$ and $\tau_N$ are uniformly distributed random permutations from $S([N]^2)$ for each $N$. We discuss the joint infinitesimal distribution of $G_N, G_N^{\sigma_N},$ and $G_N^{\tau_N}$. From \cite{mpsz1} and \cite{mvp_gaussian} one can deduce that almost surely $G_N^{\sigma_N}$ and $G_N^{\tau_N}$ are asymptotically circular and asymptotically free. Here we show that that they have zero infinitesimal distribution. Moreover we prove that $\{G_N, G_N^{\sigma_N}, G_N^{\tau_N}\}$ are asymptotically infinitesimally free.

Recall that Gaussian random matrix $G_N$ and its transpose $G_N^{\top}$ are asymptotically free \cite{MP16}. However, $G_N$ and its transpose $G_N^{\top}$ is not asymptotically infinitesimally free. Indeed,  
we find the asymptotic joint infinitesimal law of $G_N$ and $G_N^{\top}$ in Section 5. More precisely, we show that the asymptotic values (as $ N \rightarrow \infty $) of the infinitesimally free joint cumulants of
 $ G_N $ 
 and 
 $ G_N^\top $ 
 are (here each 
 $ \varepsilon_j $  is either the identity or the matrix transpose):
 \begin{align*}
 \lim_{N \rightarrow\infty}\kappa^\prime_p ( G_N^{\varepsilon_1}, G_N^{\varepsilon_2}, \dots, G_N^{\varepsilon_p}) = \left\{
 \begin{array}{ll}
 1 & \textrm{ if } p = 2m, \varepsilon_1 \neq \varepsilon_{m+1}, \varepsilon_m \neq \varepsilon_{2m}, \\
  & \textrm{ and } \varepsilon_s 
  \neq \varepsilon_{2m+1 - s} \textrm{ for } s =2, \dots, m-1\\
 0 & \textrm{ otherwise. }
 \end{array}  \right.
 \end{align*} 
 Which shows that $G_N$ and $G_N^{\top}$ are not asymptotically infinitesimally free, but have very regular joint infinitesimal free cumulants.
\section{Preliminaries}\label{section:1}

In this section, we first introduce the framework of infinitesimal freeness, then we 
review the notion of Gaussian matrices and establish the notation that we use for studying permuted Gaussian matrices. 

\subsection{Infinitesimal Free Probability}

Let us begin by recalling some notions in free probability theory. We say that $(\mathcal{A},\varphi)$ is a \emph{non-commutative probability space (ncps for short)} whenever $\mathcal{A}$ is a unital $*$-algebra over $\mathbb{C}$, and $\varphi:\mathcal{A}\to\mathbb{C}$ is a linear functional such that $\varphi(1)=1$ and $\varphi(a^*a)\geq 0$ for all $a\in\mathcal{A}$. We say unital subalgebras $\mathcal{A}_1,\dots,\mathcal{A}_s$ are \emph{free} if $a_1,\dots,a_n\in\mathcal{A}$ with $\varphi(a_j)=0$ for each $j=1,\dots,n$ and $a_j\in\mathcal{A}_{i_j}$ with $i_1\neq i_2\neq \cdots \neq i_{n-1}\neq i_n $ we have 
$$
 \varphi(a_1\cdots a_n)= 0.
$$ 

If $(\mathcal{A},\varphi)$ is a ncps with an additional linear functional $\varphi':\mathcal{A}\to\mathbb{C}$ such that $\varphi'(1)=0$ and $\varphi'(a^*)=\overline{\varphi'(a)}$ for all $a\in\mathcal{A}$, then we call the triple $(\mathcal{A},\varphi,\varphi')$ an \emph{infinitesimal probability space}. 

The natural framework of an infinitesimal probability space for algebras of random matrices was considered in \cite{BS12,shly18}.

Denote by 
$\mathbb{C}\langle X_1,\dots,X_k\rangle$
 the complex unital algebra of polynomials in $k$ non-commuting indeterminates. 
 Assume that 
 $(A_N^{(1)},\dots,A_N^{(k)})_N$
 is a sequence of $k$-tuples of random matrices 
 such that each
 $A_N^{(1)},\dots,A_N^{(k)}$
 are 
 $N\times N$ random matrices,
and  consider the sequence of linear maps $(\varphi_N)_N $ on $\mathbb{C}\langle X_1,\dots,X_k\rangle$ 
defined by
\[
\varphi_N(P)= \frac{1}{N} (E\circ Tr)\left(P(A_N^{(1)},\dots,A_N^{(k)})\right). 
\]
We say that the sequence  $(A_N^{(1)},\dots,A_N^{(k)})_N$ has asymptotic distribution 
if the limit 
$\displaystyle \varphi(P):=\lim_{N\to\infty}\varphi_N(P)$ 
exists for all 
$P\in \mathbb{C}\langle X_1,\dots,X_k\rangle$. Furthermore, if 
$$
\varphi'(P):=\lim_{N\to\infty} N[\varphi_N(P)-\varphi(P)]$$ 
exists for all $P$, then 
$\{A_N^{(1)},\dots,A_N^{(k)}\}$ is said to have the asymptotic infinitesimal distribution. We say that $\{A_N^{(1)},\dots,A_N^{(k)}\}$ are asymptotically infinitesimally free when asymptotically their joint moments with respect to to the pair of functionals $(\varphi,\varphi')$ are calculated according to the rule defined below.

 \begin{defn}
 Suppose that $(\mathcal{A},\varphi,\varphi')$ is an infinitesimal probability space. We say that unital subalgebras $\mathcal{A}_1,\dots,\mathcal{A}_s$ are infinitesimally free if for every  $ n $ and for all $a_1,\dots,a_n\in\mathcal{A}$ with $\varphi(a_j)=0$ for each $j=1,\dots,n$ and $a_j\in\mathcal{A}_{i_j}$ with $i_1\neq i_2\neq \cdots \neq i_{n-1}\neq i_n $ we have 
 \begin{eqnarray}\label{inf free cond}
 \varphi(a_1\cdots a_n)&=& 0; \nonumber \\
 \varphi'(a_1\cdots a_n)&=& \sum\limits_{j=1}^n \varphi'(a_j)\varphi(a_1\cdots a_{j-1}a_{j+1}\cdots a_n).
 \end{eqnarray}
 \end{defn}	
 It is easy to see that the condition \eqref{inf free cond} in the definition of infinitesimal freeness is equivalent to (see \cite{FN10} for a detailed explanation)
 \begin{eqnarray*}
\lefteqn{\varphi'(a_1\cdots a_n)} \\
&=& \left\{
        \begin{array}{lr}
        \varphi(a_1a_n)\varphi(a_2a_{n-1})\cdots \varphi(a_{(n-1)/2}a_{(n+3)/2})\varphi'(a_{(n+1)/2})  \\
        \ \mbox{if}\ n\ \mbox{is odd and}\ i_1=i_n,i_2=i_{n-1}\dots, i_{(n-1)/2}=i_{(n+1)/2} \\ 
        0 \ \mbox{otherwise} 
        \end{array}.
\right.
\end{eqnarray*}   

	Fix a unital algebra $\mathcal{A}$ and sequences of multilinear functionals $\{f_n:\mathcal{A}^n\to \mathbb{C}\}_{n\geq 1}$ and $\{f_n':\mathcal{A}^n\to \mathbb{C}\}_{n\geq 1}$. For a given partition $\pi\in NC(n)$, we define 
	$$
	f_{\pi}(a_1,\dots,a_n)=\prod_{V\in \pi} f_{|V|}(a_1,\dots,a_n \mid V) 
	$$
	where $(a_1,\dots,a_n|V)=(a_{i_1},\dots,a_{i_s})$ whenever $V=\{i_1<\cdots <i_s\}$.
	
    Moreover, we define $\partial f_{\pi}$ by 
    $$
    \partial f_{\pi}(a_1,\dots,a_n)=\sum_{V\in \pi} \partial f_{\pi,V}(a_1,\dots,a_n)
    $$
    where 
    $$
    \partial f_{\pi,V}(a_1,\dots,a_n)= f'_{|V|}(a_1,\dots,a_n \mid V )\prod_{W\in \pi,\ W\neq V}f_{|W|}(a_1,\dots,a_n \mid W).
    $$
    \begin{defn}
    Let $(\mathcal{A},\varphi,\varphi')$ be an infinitesimal probability space, the free cumulants $\{\kappa_n:\mathcal{A}^n\to \mathbb{C}\}_{n}$ and infinitesimal free cumulants $\{\kappa'_n:\mathcal{A}^n\to \mathbb{C}\}_n$ are sequences of multilinear functionals are defined inductively via 
    \begin{eqnarray}\label{infcumuant formula}
    \varphi(a_1\cdots a_n)&=&\sum\limits_{\pi\in NC(n)}\kappa_{\pi}(a_1,\dots,a_n), \text{ and } \nonumber \\
    \varphi'(a_1\cdots a_n)&=& \sum\limits_{\pi \in NC(n)}\partial \kappa_{\pi}(a_1,\dots,a_n). 
    \end{eqnarray}
    \end{defn}
    In \cite{FN10} the authors showed that the infinitesimal freeness can be characterized by the vanishing of $(\kappa_n,\kappa_n')_{n\geq 1}$. For reader's convenience we recall here the precise statement.
    \begin{thm}\label{vanish inf prop}
    Suppose that $(\mathcal{A},\varphi,\varphi')$ is an infinitesimal probability space, and $\mathcal{A}_1,\dots,\mathcal{A}_n$ are unital subalgebras. Then the following statement are equivalent. 
    \begin{itemize}
        \item [(i)] $\mathcal{A}_1,\dots,\mathcal{A}_n$ are infinitesimally free;
        \item [(ii)] For each $s\geq 2$ and $i_1,\dots,i_s\in [n]$ which are not all equal, and for $a_1\in\mathcal{A}_{i_1},\dots,a_s\in\mathcal{A}_{i_s}$, we have 
        $$
        \kappa_s(a_1,\dots,a_n)=\kappa'_s(a_1,\dots,a_n)=0.
        $$
    \end{itemize}
    
    \end{thm}

\subsection{Permutations for Gaussian Random Matrices}\label{notations}
In this subsection we recall some notation and relevant facts about matrices with permuted entries. In particular we review results from \cite{mvp_gaussian} and \cite{mpsz1} and we refer to these two papers for proofs.
 \begin{defn}
	By an 
	$ N \times N $
	 Gaussian random  matrix we will understand a matrix 
	  $ G = [ g_{i j}]_{ 1 \leq i, j \leq N } $  whose entries satisfy the following conditions:
	\begin{itemize}
		\item[(i)] $ g_{i, j} = \overline{ g_{j, i}}$ for all $ i, j \in [ N ] $;
		\item[(ii)]  $ \{ g_{i, j}: 1 \leq i \leq j \leq N \} $ is a family of independent, identically distributed complex $($if $i \neq j)$ or real $($if $ i = j )$ Gaussian random variables of mean 0 and variance 
		$ \displaystyle \frac{1}{N} $.
	\end{itemize} 
\end{defn}

If $ n $ is a positive integer, we shall denote by $ [ n ] $ the ordered set 
$ \{ 1, 2, \dots, n \} $. 
The set of pair partitions of
$ [ n] $ is denoted by $ P_2 (n) $. In particular, if 
$ n $ 
is odd, then 
$ P_2 (n) = \emptyset $.

The set of all permutations of $ [ N ] \times [ N ] $ will be denoted by 
$ \mathcal{S} ( [ N ]^2) $. For $(i,j)\in [N]\times [N]$, we denote by $\top$ the transpose, i.e. we have $\top(i,j)=(j,i)$. 

In addition, for a given $\sigma\in \mathcal{S}([N]^2)$, we define 
$$
\sigma(A)=\{\sigma(i,j) : (i,j)\in A\} \text{ for }A\subset [N]\times [N].
$$
Recall that for 
an
$ N \times N $ Gaussian random matrix $ G_N $ and
$ \sigma \in \mathcal{S} ([ N ]^2 ) $,
we denote 
$ G_N^\sigma $ to be 
the random matrix whose 
$ (i, j)$-entry equals the
 $ \sigma (i, j)$-entry of 
$ G_N $; 
i.e.
$$ [G_N^{ \sigma } ]_{ i, j} =  g_{\sigma(i, j)}. 
$$

 Assume that for each positive integer $ N $ and each
  $ k = 1, 2, \dots, m $, 
 $ \sigma_{ k, N } $ 
 is a permutation from 
 $ \mathcal{S} ([ N ]^2) $.
 With these notations, we have that
 \begin{align*} E \circ \tr 
 \Big( 
 G_N^{ \sigma_{1, N}}
  \cdot G_N^{ \sigma_{2, N}} 
 \cdots 
 G_N^{ \sigma_{m, N}} 
 \Big) 
 &=\\  \sum_{ 1 \leq i_1, \dots, i_m \leq N } &
 \frac{1}{N} E \big( 
 [ G_N^{ \sigma_{1, N}}]_{ i_1, i_2} \cdot
 [ G_N^{ \sigma_{2, N}} ]_{ i_2, i_3} 
 \cdots 
 [  G_N^{ \sigma_{m, N}} ]_{ i_m, i_1}
 \big) \\
 \end{align*}
 As shown in \cite{mvp_gaussian}, using Wick's formula (see \cite{jason}) for the right-hand side of the equation above, with the identification 
 $i_{m+1} = i_1$,
  we obtain
\begin{align}\label{eq:1}
 E \circ \tr 
\Big( 
G_N^{ \sigma_{1, N}}
\cdot G_N^{ \sigma_{2, N}} 
\cdots 
G_N^{ \sigma_{m, N}} 
\Big) 
= 
\sum_{ \pi \in P_2(m)}  \mathcal{V}_{ \overrightarrow{\sigma_N} } (\pi ) 
\end{align}
where we use short-hand notation 
 $ \overrightarrow{\sigma_N } =
  ( \sigma_{1, N}, \sigma_{2, N}, \dots, \sigma_{m, N}) $ and
\[ \mathcal{V}_{ \overrightarrow{\sigma_N}} ( \pi)=
\frac{1}{N}  \sum_{ 1 \leq i_1, \dots, i_m \leq N } \prod_{ (k, l) \in \pi } 
E \big( [ G_N^{ \sigma_{k, N}}]_{ i_k, i_{k+1}} \cdot
[ G_N^{ \sigma_{l, N}} ]_{ i_l, i_{l+1}} \big).
\] 
Moreover, since 
$ G_N $ 
is Gaussian,
 $ \displaystyle  E \big( [ G_N^{ \sigma_{k, N}}]_{ i_k, i_{k+1}} \cdot
 [ G_N^{ \sigma_{l, N}} ]_{ i_l, i_{l+1}} \big) 
 = \frac{1}{N}\delta_{ \sigma_{k, N} ( i_k, i_{ k + 1})}^{ \top\circ \sigma_{l, N} (i_l, i_{ l+ 1})}.
 $
Since it is important to keep track of which indices are equal to each other, it is standard to encode this with pair partitions.  Therefore for a given $\pi\in P_2(m)$, we denote by $\mathcal{A}_{ \pi, \overrightarrow{\sigma}_N}$ such sequences of indices which respect $\overrightarrow{\sigma}_N$ and $\pi$ in the sense that they contribute in the sum above. To make the above precise we will view any pair partition as a permutation, where each block becomes a cycle, and we will write $\pi(k)$, to mean the image of $k$ under the permutation (induced by) $\pi$. With the identification $i_{m+1}=i_1$ we define

 \begin{align*}
 \mathcal{A}_{ \pi, \overrightarrow{\sigma}_N} 
  = \big\{ (i_s, j_s)_{ s \in [ m ]} :\  
  j_k=i_{k+1}    \textrm{ and }
   \sigma_{ k, N} (i_k, j_k) = 
  \top \circ & \sigma_{ \pi(k), N }  ( i_{ \pi(k)}, j_{ \pi(k)} ) \\
 &  
  \textrm{ \ \ \  \ \ for all }k \in [ m ]
   \big\}
 \end{align*}
 and
 \[ \mathfrak{a}_{ \pi, \overrightarrow{\sigma}_N} = 
 \log_N | \mathcal{A}_{ \pi, \overrightarrow{\sigma}_N} | - ( \frac{m}{2} + 1).
  \]
 Then \eqref{eq:1} can be rewritten as
 \[
  E \circ \tr 
 \Big( 
 G_N^{ \sigma_{1, N}}
 \cdot G_N^{ \sigma_{2, N}} 
 \cdots 
 G_N^{ \sigma_{m, N}} 
 \Big) 
 = \sum_{\pi \in P_2(m)} | \mathcal{A}_{ \pi, \overrightarrow{\sigma}_N} |N^{-(\frac{m}{2}+1)}
 = 
 \sum_{ \pi \in P_2(m)} N^{ \mathfrak{a}_{ \pi, \overrightarrow{\sigma}_N}}.
 \]

  For 
 $ B = \{ \beta_1, \beta_2, \dots, \beta_r \} $
 an ordered subset of 
 $ [ m ] $  
 and 
 $ (i_s, j_s)_{s\in [ m ]} \in N^{2m} $
 we will be interested in subsequences $ (i_s, j_s)_{s\in B} $ 
 i.e.
 $ (i_{\beta_1}, j_{\beta_1}, \dots, i_{\beta_r}, j_{\beta_r}) $ which can be extended to an element from $\mathcal{A}_{ \pi, \overrightarrow{\sigma}_N} $, thus we define
 \begin{align*}\mathcal{A}_{ \pi, \overrightarrow{\sigma}_N} (B)
 = \big\{ (i_s, j_s)_{s\in B } :\ 
  \textrm{ there exists some } &(i_s, j_s)_{s \in [ m ]\setminus B } \\  & \textrm{ such that }
   (i_s, j_s)_{s \in [ m ] } \in \mathcal{A}_{\pi, \overrightarrow{\sigma_N} }
   \big\}
 \end{align*}
 and let
 \[
 \mathfrak{a}_{\pi, \overrightarrow{\sigma}_N} (B)
 = 
 \log_N | \mathcal{A}_{ \pi, \overrightarrow{\sigma}_N} (B) | - | B |  + \frac{1}{2}| B^2\,_{| \pi} | - 1 .
 \]
 where $B^2\,_{| \pi}$ can be understand as the range of $\pi$ i.e. we have
 $ B^2\,_{| \pi} = \{ (k, l )  \in B \times B :\  \pi(k) = l \} $.
 
 The following results are shown in \cite{mvp_gaussian} (see  Lemmas 2.2 --2.4).

\begin{lemma}\label{lemma:01} ${}$
	
\begin{enumerate}
	\item [(i)] If $ k \in B $, then 
$\mathfrak{a}_{ \pi, \overrightarrow{\sigma}_N} (B \cup \{ k -1 \} )
\leq
\mathfrak{a}_{ \pi, \overrightarrow{\sigma}_N} (B )	$ 
and \\
${}$ \hspace{5.5cm} $ \mathfrak{a}_{ \pi, \overrightarrow{\sigma}_N} (B \cup \{ k + 1 \} )
\leq 
\mathfrak{a}_{ \pi, \overrightarrow{\sigma}_N} (B ) $. 

In particular, $\mathfrak{a}_{ \pi, \overrightarrow{\sigma}_N} (B ) \geq \mathfrak{a}_{ \pi, \overrightarrow{\sigma}_N}$ for all $B\subset [m]$.
	\item[(ii)]
	If $ k \in B $, then 
	$ 
	\mathfrak{a}_{\pi, \overrightarrow{\sigma}_N} (B \cup \{ \pi(k)\} )
	\leq
	\mathfrak{a}_{\pi, \overrightarrow{\sigma}_N} (B).
	$
	\item[(iii)] If $ \{k, \pi (k)\} \cap B = \emptyset $
	and 
	$ \{ k-1, k+1 \} \subseteq B $, 
	then \\
${}$ \hspace{5cm}	$ \mathfrak{a}_{\pi, \overrightarrow{\sigma}_N} (B \cup \{ k \} )
	 \leq 
\mathfrak{a}_{\pi, \overrightarrow{\sigma}_N} (B ) -1  . $
	\end{enumerate}
\end{lemma}

\subsection{Basic probabilistic tools}

We will repeatedly use Borel--Cantelli lemma in order to prove almost sure convergence of some statistics of random uniform permutations. We show several similar lemmas, which nevertheless have substantially different proofs. We prove each lemma separately, in each case invoking the following basic fact.

\begin{lemma}\label{lem:BC}
    Let $(X_N)_N$ be a sequence of non-negative random variables, for $X_N\stackrel{1}{\to}0$ it suffices to show that for any $\varepsilon>0$ there exist $\delta>0$ and $C>0$ such that $\mathbb{P}(X_N>\varepsilon)\leq \tfrac{C}{N^{1+\delta}}$.
\end{lemma}


\section{The generic permuted Gaussian random matrix has zero infinitesimal distribution}\label{Section:3}

As explained in the previous section, in order to understand the joint moments of a Gaussian matrix with permuted entries it suffices to understand $\mathcal{V}_{ \overrightarrow{\sigma_N} } (\pi )$ for any pairing $\pi$. In this section we will consider random permutations, and we will consider asymptotic behaviour of $\mathcal{V}_{ \overrightarrow{\sigma_N} }$, where we assume that $(\sigma_N)_N$ is a sequence of uniformly random permutations with $\sigma(N)\in \mathcal{S} ( [ N ]^2 )$ for each $N\geq 1$. 
Note that in $\mathcal{V}_{ \overrightarrow{\sigma_N} } (\pi )$ we integrate with respect to the distribution of entries of the matrix, so the almost sure limits mentioned below are with respect to the sequence of random permutations only. Since a Gaussian matrix after a random permutation most likely is not self-adjoint we have to take care of complex-conjugate together with the matrix permutation. Observe that $(G^\sigma)^*_{(i,j)}=\overline{G_{\sigma(j,i)}}$ and this is exactly the same as $G^{\top \circ \sigma_N \circ \top}_{(i,j)}$, which motivates the notations we introduce below. The goal of this section is to prove the following theorem.

\begin{thm}\label{thm:3:1}
	Let
	 $ \big( \sigma_{N } \big)_N $
	  be a sequence of uniformly random permutations from 
$ \mathcal{S} ( [ N ]^2 ) $ and 
$ \varepsilon(1), \dots, \varepsilon(m) \in \{1,\ast\}$. Define
\[ 
\sigma_{k, N} = \left\{  
\begin{array}{ll}
\sigma_N & \textrm{ if } \varepsilon (k) = 1\\
\top \circ \sigma_N \circ \top & \textrm{ if } \varepsilon (k) = \ast . 
\end{array} \right.
\]

With the notations from Section \ref{notations}, almost surely we have that 
$ \mathcal{V}_{ \overrightarrow{\sigma_N}} ( \pi )  = o(N^{-1}) $ 
 unless 
 $ \pi $ 
 is non-crossing and
  $ \varepsilon ( k ) \neq \varepsilon ( \pi(k) ) $
 for all $ k \in [ m ] $. 
\end{thm}

In order to prove Theorem \ref{thm:3:1} we need to establish several technical results.

\begin{lemma}\label{lemma:01:1}
Let 
$ \big(  \sigma_N  \big)_N $ 
be a sequence of uniformly random permutations, with each 
$ \sigma_N $ 
from 
$ \mathcal{S}( [ N ]^2 ) $.
For any constant $ \theta > 0 $ we have the following almost sure limits:
\begin{enumerate}
\item[(i)] $\displaystyle \lim_{N \rightarrow \infty}
N^{-(\frac{1}{2} + \theta)} \cdot
 \big|\{ (i, j) \in [ N ]^2:\
 \sigma_N (i, j) = \top \circ \sigma_N (j, i)   \}\big| = 0
$
\item[(ii)] $ \displaystyle 
\lim_{N \rightarrow \infty}
N^{-\theta} \cdot
\sup_{ 1 \leq i \leq N } \big| \{ (j, k) \in [ N ]^2:\ 
\sigma_N ( i, j) \in \{  \top \circ \sigma_N (j, k), \sigma_N \circ \top (j, k) \} \} \big| = 0.
$
\end{enumerate}
\end{lemma}

\begin{proof}

For part (i),  for each 
$i, j \in [ N ] $,
denote by 
$ I_{ i, j} $
the random variable on 
$ \mathcal{S} ( [N]^2)	$
given by
  $ I_{i, j} ( \sigma) = \left\{
\begin{array}{ll}
1 & \textrm{ if } \sigma (i, j)= \top \circ \sigma (j, i) \\
0 & \textrm{ if } \sigma (i, j)\neq \top \circ \sigma (j, i)
\end{array}
\right. $
and let 
$ \displaystyle  X_N = \sum_{ i, j = 1}^N I_{i, j}.	$	


Using Markov's inequality, we have that 
$ \displaystyle \mathbb{P} ( N^{-(\frac{1}{2} + \theta )} X_N > \varepsilon ) \leq \frac{ \mathbb{E} (X_N^2)} { \varepsilon^2 N^{(1 + 2 \theta)} } $
so from Lemma \ref{lem:BC} it suffices to show that 
$ \mathbb{E} (X_N^2) $
is bounded. 

Since 
$ I_{i, j} = I_{j, i}= I^2_{i, j} $,
we have that 
\begin{align*}
\mathbb{E} (X_N^2) = \mathbb{E} \big( \big(\sum_{i, j=1}^N I_{i, j} \big)^2 \big) = 
2\sum_{i, j=1}^N \mathbb{E} (I_{i, j}) + 
\sum_{\substack{ 1 \leq i, j, k, l \leq N \\(i, j) \notin \{ (k, l), (l, k)\} } } 
\mathbb{E} ( I_{i, j} I_{k, l}).
\end{align*}

 To estimate 
$ \mathbb{E} ( I_{i, j}) $, 
note that if $j\neq j$ and
$ I_{i, j}(\sigma) \neq 0 $
the value of 
$ \sigma(i, j) $
uniquely determines the value of
$ \sigma( j, i) $;
there are
$ N^2 $ 
possible choices for
$\sigma(i, j) $
and hence 
$ \sigma (j, i)$
and
$(N^2 -2)!$
possibilities for the rest of the values of $\sigma $.
For elements on the diagonal, i.e.
  when $i=j$, we have that
$ I_{i, j}(\sigma)  \neq 0 $
 if and only if $ \sigma(j, i) $ is also on the diagonal. Therefore
\begin{align*}
\sum_{i, j=1}^N \mathbb{E}(I_{i, j})=\sum_{i\neq j}\mathbb{E}(I_{i, j})+\sum_{i}\mathbb{E}(I_{i, i}) \leq 2 {N \choose 2} \frac{N^2  \cdot (N^2 -2)! }{N^2!}+N \frac{1}{N}  \xrightarrow[{N \rightarrow \infty}]{} 2.
\end{align*}

Similarly, if 
$(i,j) \notin \{ (k, l), (l, k) \} $,
for 
$ I_{i, j} (\sigma) I_{k, l}(\sigma) \neq 0 $
there are 
$ N^2 $
possible choices for 
$ \sigma( i, j) $,
each uniquely determining at most one possible choice for 
$ \sigma(j, i) $. 
Then there are 
$ N^2 - 2 $
possible choices for
$ \sigma(k, l) $,
each determining at most one possible choice for
$ \sigma(l, k) $
and
$ N^2 - 4 $ 
possible choices for the rest of the values of the permutation
$\sigma $.
Hence
\begin{align*}
\sum_{\substack{ 1 \leq i, j, k, l \leq N \\(i, j) \notin \{ (k, l), (l, k)\} } } 
\mathbb{E} ( I_{i, j} I_{k, l}) 
\leq N^2 (N^2 -2) \frac{N^2 \cdot (N^2 -2)\cdot (N^2 - 4)! }{N^2!}
\xrightarrow[{N \rightarrow \infty}]{} 1 .
\end{align*}


For part (ii) fix $\varepsilon>0$ and observe that using  exchangeability of rows and subadditivity,  
we have that for any fixed
$ a \in [ N ] $
\begin{align*}
\mathbb{P} \big(
\sup_{ 1 \leq i \leq N }  
\big| \{ (j, k) \in  [ N ]^2  & :\ 
\sigma_N ( i, j) = \top \circ \sigma_N (j, k) \}  \big| > 
N^{\theta}  \cdot \varepsilon 
\big)\\
\leq & \sum_{ i =1}^N 
\mathbb{P} \big( 
\big| 
\{ (j, k) \in [ N ]^2:\ 
\sigma_N ( i, j) = \top \circ \sigma_N (j, k) \}
\big| > N^{\theta}  \cdot \varepsilon 
\big)\\
= & N \cdot \mathbb{P} \big( 
\big| 
\{ (j, k) \in [ N ]^2:\ 
\sigma_N ( a, j) = \top \circ \sigma_N (j, k) \}
\big| > N^{\theta}  \cdot \varepsilon 
\big).
\end{align*}

Denoting 
$ Y_N = \sum_{1 \leq i, j \leq N} F_{i, j} $
where each
$ F_{i, j} $ 
the random variable on
$ \mathcal{S}([ N ]^2 )$
given by
  $ F_{i, j} ( \sigma) = \left\{
\begin{array}{ll}
1 & \textrm{ if } \sigma (a, i)= \top \circ \sigma (i, j) \\
0 & \textrm{ if } \sigma (a, i)\neq \top \circ \sigma (i, j)
\end{array}
\right. $
and applying Markov's inequality we obtain
\begin{align*}
\mathbb{P} \big(
\sup_{ 1 \leq i \leq N } 
\big| \{ (j, k) \in [ N ]^2:\ 
\sigma_N ( i, j) = \top \circ \sigma_N (j, k) \}  \big| > 
N^{\theta}  \cdot \varepsilon 
\big)
\leq
N 
\frac{\mathbb{E} 
\big(  Y_N^K\big)} { ( N^{\theta} \varepsilon )^K},
\end{align*}
for any positive integer $ K $, in particular for $ K >\frac{3}{\theta}$.

Since
$ F_{i, j}^2 = F_{i, j} $, 
we have that 
$\mathbb{E} 
	\big(  Y_N^K\big) = \displaystyle \sum_{s=1}^K
	\mathbb{E} \big(
	\sum F_{i_1, j_1} \cdots F_{i_s, j_s} \big) $ 
	where the second sum goes over 
	$ \{ (i_1, j_1, \dots, i_s, j_s) \in [N]^{2s}:\ 
	(i_1, j_1), \dots, (i_s, j_s) \textrm{ distinct} \}.
$	
So Lemma \ref{lem:BC} implies that
it suffices to show the uniform  boundedness in $ N $ of the expression
$ \displaystyle
 \mathbb{E} \big(
  \sum_D F_{i_1, j_1} \cdots F_{i_s, j_s} \big) $
where 
$ s $
is some positive integer and 
\begin{align*}
 D = \{ (i_1, j_1, \dots, i_s, j_s)\in [ N ]^{2s} :\ (i_1, j_1), \dots, (i_s, j_s) \textrm{ distinct } \}.
\end{align*} 

 We write D as a disjoint union
  $D=D_1\sqcup D_2$, 
  where
  \begin{align*}
  D_1 & = \{ (i_1, j_1, \dots, i_s, j_s) \in D : \ 
  (i_k, j_k) =(a, a) \textrm{ for some } k\in [s]\} \\
  D_2 & =  \{ (i_1, j_1, \dots, i_s, j_s) \in D : \ 
  (i_k, j_k) \neq (a, a) \textrm{ for all } k\in [s]\}.
  \end{align*}
 \begin{itemize}
 	\item [$\bullet$] $(i_1,j_1,\ldots,i_s,j_s)\in D_1$ if there is $1\leq k\leq s$ such that $(i_k,j_k)= (a,a)$ (in particular, since the pairs of indices in $ D $ are distinct, there can be at most one such $k $).
 	\item[$\bullet$] $(i_1,j_1,\ldots,i_s,j_s)\in D_2$ if for all $1\leq k \leq s$ we have $(i_k,j_k)\neq(a,a)$.
 \end{itemize} 

Under the condition
$F_{i_1,j_1}\cdots F_{i_s,j_s}\neq 0$, 
for each
$(i_1,j_1,\ldots,i_s,j_s)\in D_2$
there are at most
 $\frac{N^2 ! }{( N^2 - s)!} $
possible choices for the 
$ s $-tuple
$(\sigma(a, i_1) \dots, \sigma(a, i_s)) $
 each of them determining at most one possible choice for the $s$-tuple  
 $(\sigma(i_1, j_1) \dots, \sigma( i_s, j_s)) $
 and 
 $ (N^2 - 2s)! $ 
 possible choices for the rest of values of
 $\sigma $.
 Furthermore, for each 
 $(i_1,j_1,\ldots,i_s,j_s)\in D_1$
 there are at most $N$ choices for the value of $\sigma(a,a)$ and at most
$\frac{(N^2-1)!}{(N^2-s)!}$
for the values of all other 
$ \sigma(a, i_k) $ 
such that 
$F_{i_1,j_1}\cdots F_{i_s,j_s}\neq 0$;
each such choice determines at most one possible choice for 
$(\sigma(i_1, j_1) \dots, \sigma( i_s, j_s)) $
and at most
$ (N^2 - 2s)! $ 
possible choices for the rest of values of
$\sigma $. 
Also, since $(i_k,j_k)=(a, a)$ for a unique 
$k \in [s]$,
 we have at most $N^{2s-1}$ choices for the $2s$-tuple $(i_1,j_1,\ldots,i_s,j_s)\in D_2$.

   Therefore
   \begin{align*}
 \mathbb{E} \big(
 \sum_D F_{i_1, j_1} \cdots F_{i_s, j_s} \big) \leq &
N^{2s} \cdot  \frac{1}{N^2!} \cdot \frac{N^2 ! }{( N^2 - s)!}
 \cdot 1 
 \cdot (N^2 - 2s)!
 \\
 +& N^{2s-1}\cdot \frac{1}{N^2!} \cdot \frac{N (N^2-1)!}{(N^2-s)!}\cdot 1\cdot (N^2-2s)!
 \xrightarrow[ N \rightarrow \infty ]{} 1,
   \end{align*}
   where $N^{2s}$ comes from estimating the number of terms in the sum,
   hence the conclusion.
The argument of the case $\sigma_N(i.j)=\sigma_N\circ \top(j,k)$ is similar. 
\end{proof}

\begin{lemma} \label{lemma:1:2:3}
For each positive integer
	$ N $ 
let 
	$\phi_N $
be a fixed map from 
$ [ N ] \times [ N ] $
to 
$ [ N ] $. 
Suppose that each
$ \omega_{s, N} $, $ s \in \{1, 2, 3\} $
is either the transpose for each
$ N $, or the identity for each $N $
and denote
$ \sigma_{s, N} = \omega_{s, N} \circ \sigma_N \circ \omega_{s, N} $.	
With this notations, for any $ \theta > 0 $ the following relations hold true almost surely:
\begin{enumerate}
	\item[(i)] $ \displaystyle 
\lim_{N \rightarrow \infty}	
N^{-(1+ \theta)} \cdot 
\big| \{ (i, j, k, l) \in [ N ]^4 :\ \sigma_N (i, j) =
 \top \circ \sigma_{1, N} 
( k, l) \textrm { and } \\
{}\hspace{6.5cm} 
\sigma_{2, N}  (j, k) = \top \circ \sigma_{3, N} ( l, i)
\}\big| = 0.
	$
	\item [(ii)]
	$ \displaystyle \lim_{N\rightarrow\infty}
	N^{-(\frac{5}{2} + \theta)}\cdot 
	\big|\big\{ (i, j, a, b)\in [ N ]^4  : \  \sigma_{N}( a, \phi_N (\sigma_N (i, j))) = \top \circ \sigma_{1, N} (b, i) \big\} \big|=0$.
\end{enumerate}
\end{lemma}

\begin{proof}
To simplify the writing,  $ N $ shall be omitted within this proof by writting 
$ \sigma $, $\sigma_s $, $\omega_S $ for $\sigma_N $,
$ \sigma_{s, N}$, $\omega_{s, N} $ respectively (where $1\leq s \leq 4 $ ).

 For part (i), note first that if 
 $ i = j $, 
 then there are at most 
 $ N $
  possible choices for 
  $ i $, 
  each determining at most one possible choice for 
  $ k, l $,
  hence
  \[ \big| \{ (i, j, k, l) \in [ N ]^4 :\  i=j, \sigma_N (i, j) =
  \top \circ \sigma_{1, N} 
  ( k, l),
  \sigma_{2, N}  (j, k) = \top \circ \sigma_{3, N} ( l, i)
  \}\big|=o(N).
  \]
 Similar relations hold true for   
 $ k= l$, for $ i= l $ and for  $ j = k $. 
 Furthermore, if 
 $ i = k $ 
 and 
 $ l = j $,
 then the equality
 $ \sigma (i, j) = \top \circ \sigma_{1}( k, l) $ 
 gives either
  $ i= j $
   if 
   $\omega_{1} = \id_N $
   or
   $ \sigma (i, j) = \top \circ \sigma (i, j) $
   if 
$\omega_{1} $  is the transpose. 
In the last case
 there are again at most 
$ N $ 
possible choices for 
$ (i, j)$
as it is in the preimage  of the diagonal under 
$ \sigma_N $.

Therefore it remains to prove the statement for
$ (i, j, k, l) \in D $,
where
\begin{align*}
D = \big\{  (i, j, k, l) \in [ N ]^4 :\  (i, j),\, \omega_{1} (k, l),\, \omega_{2} (j, k),\,
\omega_{3} (l, i)
\textrm{ are distinct} \big\}.
\end{align*}

  For each
 $ ( i, j, k, l) \in D $,
 consider a mapping
 $ I_{i, j, k,l} $
 on
 $ \mathcal{S}([N]^2) $
 given by  
 
 \begin{align*}
  I_{i, j, k, l} ( \sigma) =
 \begin{cases}
 1 & \mbox{ if  condition ($\ast$)  holds true } \\
 0 & \mbox{ otherwise, }
 \end{cases}
 \end{align*}
where condition ($\ast$) is that
$\sigma (i, j) =
\top \circ \sigma_{1} 
( k, l) $ 
 and
$
\sigma_{2}  (j, k) = \top \circ \sigma_{4} ( l, i) $.


From Markov's inequality, we have
 \[ \mathbb{P} (
 \sum_{ (i, j, k, l) \in D } I_{i,j, k, l}
  > \varepsilon N^{1+\theta} ) \leq \varepsilon^{-1} \cdot
 \mathbb{E} \big(
  \sum_{ (i, j, k, l) \in D } I_{i,j, k, l} 
   \big)/N^{1+\theta},\]
where the (not displayed) argument of $I_{i,j, k, l}$ is an uniformly random permutation on $\mathcal{S}([ N]^2 )$. From Lemma \ref{lem:BC} it sufficces to show that the expectation above is bounded.

For 
$ I_{i,j, k, l}( \sigma) \neq 0 $,
there are 
$ N^2 $ 
possible choices for 
$ \sigma (i, j)$,
each giving at most one possible choice for 
$ \sigma ( \omega_{1, N} (k, l)) $
and 
$ N^2 - 2 $
choices for 
$ \sigma ( \omega_{2}(j, k))$,
each of them determining at most one possible choice for
$ \sigma( \omega_{3}(l, i)) $;
finally, there are at most 
$ (N^2 -4)! $ 
possible choices  for the rest of the values of 
$ \sigma $. 
Therefore
\begin{align*}
\mathbb{E} \big( 
\sum_{ (i, j, k, l) \in D } I_{i,j, k, l} 
\big) < 
N^4 \cdot \frac{N^2 \cdot (N^2-2) \cdot (N^2 -4)!}{N^2!} 
\xrightarrow[ N \rightarrow \infty]{} 1,
\end{align*}
hence the conclusion.

For part (ii), let us denote by 
($\ast\ast$) the relation
$ \sigma (a, \phi_N ( \sigma (i, j))) = \top \circ \sigma_{1} (b, i)$.

Remark that it suffices to show the property for tuples 
$ (i, j, a, b) $ 
such that 
$ a, b, \notin \{ i, j\}$. 
If 
$ a \in \{i, j\} $, 
then there are 
$ N^2 $ 
possible choices for the triple 
$ (i,j, a)$,
each determining, via condition ($\ast\ast$), at most one possible value for 
$ b $. Similarly for 
$ b \in \{i, j\}$.

Let 
$ D_1 = \{ (i, j, a, b) \in [ N ]^4:\  a, b \notin \{i, j\}\}$.
As before, using the Borel-Cantelli Lemma, it suffices to show that for each
$ \varepsilon > 0 $,
\begin{align*}
\sum_{N =1}^\infty 
\mathbb{P} \big(
\big| \{ (i,j, a, b ) \in D_1 :\  \textrm{ ($\ast\ast$) holds true} 
\}
\big| > \varepsilon \cdot  N^{\frac{5}{2}+ \theta}
\big)< \infty,
\end{align*}
which holds true if 
\begin{align*}
N^{1+ 2 \theta} \cdot \mathbb{P} \big(
\big| \{ (i,j, a, b )\in D_1 :\  \textrm{ ($\ast\ast$) holds true}
\}
\big| > \varepsilon \cdot  N^{\frac{5}{2} + \theta }
\big)
\end{align*}
is uniformly bounded in 
$ N $. 

For each $(i, j, a, b) \in D_1 $
consider
 a mapping
 $ F_{i, j, a,b,} $
  on
$ \mathcal{S}([ N]^2 )$ 
given by
\begin{align*}
 F_{i, j, a, b} ( \sigma)
= 
\begin{cases}
\begin{array}{ll} 1& \textrm{ if ($\ast\ast$) holds true}\\
0 & \textrm{otherwise}.
\end{array}
\end{cases}
\end{align*}
Markov's inequality gives that 
\begin{align*}
\mathbb{P} \big(
\big| \{ (i,j, a, b )\in D_1 :\  \textrm{ ($\ast\ast$) holds true}
\}
\big| > \varepsilon \cdot  N^{\frac{5}{2}+\theta}
\big) 
\leq 
\frac{\mathbb{E} \big( \displaystyle ( \sum_{ (i, j, a, b)\in D_1 } 
	F_{i, j, a, b} )^2  \big) }{\varepsilon^2 N^{5 + 2 \theta}}
\end{align*}
where the implicit argument of $F_{i, j, a, b}$ is a uniformly random permutation on $ \mathcal{S}([ N]^2 )$. So by Lemma \ref{lem:BC} it suffices to show that 
\begin{align}\label{bdd:03}
N^{-4} \cdot \mathbb{E} \big( \displaystyle ( \sum_{ (i, j, a, b)\in D_1 } 
F_{i, j, a, b} )^2  \big) \textrm{ is uniformly bounded in
	$ N $. }
\end{align}

First, note that 
if 
$ b \neq  b^\prime$,
then
$ F_{i, j, a, b} \cdot F_{i, j, a, b^\prime}= 0$.
since condition ($\ast\ast$) implies that 
$ \top \circ \sigma_{1} (b, i)
 = \sigma ( a, \phi_N( \sigma(i, j))) 
 = \top \circ \sigma_{1}(b^\prime, i)$.
 Similarly, if
 $ a \neq a^\prime $,
 then
 $ F_{i, j, a, b} \cdot F_{i, j, a^\prime, b}= 0$.
 
 Next, if
 $ (i, j) \neq (i^\prime, j^\prime)$
 but
 $ (a, \phi_N ( \sigma_N (i, j)))
 = 
 (a^\prime, \phi_N ( \sigma (i^\prime, j^\prime)))$
 then
 ($\ast\ast$) gives that 
 $(b, i) = (b^\prime, i^\prime)$.
 
 Denote
 \begin{align*}
 D_2 &=
  \{ (i, j, a, b, a^\prime, b^\prime):\ 
 a\neq a^\prime, b \neq b^\prime \textrm{ and }
 (i, j, a, b), (i, j, a^\prime, b^\prime) \in D_1 \}\\
 D_3 & =
 \{ (i, j, j^\prime, a, b):\
  (i, j, a, b), (i, j^\prime, a, b) \in D_1 \}\\
  D_4 & = 
  \{ (i, j, i^\prime, j^\prime, a, b, a^\prime, b^\prime):\ (i, j) \neq (i^\prime, j^\prime), (a, b) \neq (a^\prime, b^\prime) \textrm{ and }\\
  & \hspace{7.3cm} 
  (i, j, a, b), (i^\prime, j^\prime, a^\prime, b^\prime)
  \in D_1 \}
 \end{align*}

Utilizing the observations above and that
$ F_{i, j, a, b} = F_{i, j, a, b}^2 $,
we have that
\begin{align*}
\mathbb{E} \big( 
 ( \sum_{ (i, j, a, b)\in D_1 } 
F_{i, j, a, b} )^2  \big)
= \mathbb{E} & \big( \sum_{ D_1 } F_{i, j, a, b} \big)
+ 
\mathbb{E}\big(
\sum_{D_2} F_{i,j, a, b} \cdot F_{i, j, a^\prime, b^\prime} 
\big)\\
& + 
\mathbb{E}\big(
\sum_{D_3} F_{i,j, a, b} \cdot F_{i, j^\prime, a, b} 
\big)
+ 
\mathbb{E}\big(
\sum_{D_4} F_{i,j, a, b} \cdot F_{i^\prime, j^\prime, a^\prime, b^\prime} 
\big).
\end{align*}

For 
$ F_{i,j, a, b} (\sigma_N)= 1 $,
there are at most 
$ N^2 $ 
possible choices for 
$ \sigma(i, j) $, 
each giving at most 
$ N^2 -1 $ 
possible choices for
$ \sigma ( a, \phi_N ( \sigma(i, j)))$,
each giving at most one possible choice for
$ \sigma ( \omega_{1}(b, i)) $
and at most
$ (N^2 -3)! $
possible choices for the rest of values of
 $ \sigma $.
Therefore
\begin{align}\label{sum:1}
N^{-4}  \cdot   \mathbb{E} \big( \sum_{ D_1 } F_{i, j, a, b} \big)\leq N^{-4} \cdot N^4\cdot \frac{ N^2 \cdot (N^2 -1 ) \cdot 1 \cdot (N^2 -3)! }{N^2!}
\xrightarrow[ N \rightarrow\infty]{} 0.
\end{align}

For 
$ (i, j, a, b, a^\prime, b^\prime) \in D_2 $
and
$ F_{i,j, a, b} \cdot F_{i, j, a^\prime, b^\prime} ( \sigma) = 1 $,
there are 
$ N^2 $ possible choices for 
$ \sigma(i, j) $,
each giving less than 
$ N^4 $ 
possible choices  for
$ \sigma ( a, \phi_N ( \sigma(i, j)))$,
$ \sigma ( a^\prime, \phi_N ( \sigma(i, j)))$
each giving at most one choice for 
$ \sigma ( \omega_{1}(b, i)) $,
$ \sigma ( \omega_{1}(b^\prime, i)) $
and
$(N^2 - 5 )! $
possible choices for the rest of values of 
$ \sigma $. 
Therefore
\begin{align}\label{sum:2}
N^{-4} \cdot
 \mathbb{E}\big(
 \sum_{D_2} F_{i,j, a, b} \cdot F_{i, j, a^\prime, b^\prime} 
 \big)<
  N^{-4} \cdot N^6 \cdot \frac{N^2 \cdot N^4 \cdot (N^2 -5)!}{N^2!}
 \xrightarrow[ N \rightarrow\infty]{}0.
\end{align}

For 
$(i, j, j^\prime, a, b) \in D_3 $
and
$  F_{i,j, a, b} \cdot F_{i, j^\prime, a, b} = 1 $,
there are
$ N^2 $ 
possible choices for 
$ \sigma(i, j) $, 
each giving less than
$ N^2 $ 
possible choices for
$ \sigma ( a, \phi_N ( \sigma(i, j)))$,
each giving at most one possible choice for
$ \sigma ( \omega_{1}(b, i)) $
and at most
$ (N^2 -3)! $
possible choices for the rest of values of
$ \sigma $. 
Therefore
\begin{align}\label{sum:3}
N^{-4} \cdot \mathbb{E}\big(
\sum_{D_3} F_{i,j, a, b} \cdot F_{i, j^\prime, a, b} 
\big) 
< 
N^{-4} \cdot N^5 \cdot \frac{N^4 \cdot (N^2 -3)!}{N^2!}
\xrightarrow[N \rightarrow\infty]{}0.
\end{align}

For 
$(i, j, i^\prime, j^\prime, a, b, a^\prime, b^\prime)
\in D_4 $
together with
 $F_{i,j, a, b} \cdot F_{i^\prime, j^\prime, a^\prime, b^\prime} = 1 $
 there are at most 
 $ N^2 $ possible choices for 
 $ \sigma(i, j) $,
 each giving less than
 $ N^2 $ possible choices for 
$ \sigma ( a, \phi_N ( \sigma(i, j)))$,
each giving one possible choice for 
$ \sigma( \omega_{1} (b, i)) $, 
then less than
$N^4 $ 
possible choices for 
$ \sigma(i^\prime, j^\prime)$,
$ \sigma ( a^\prime, \phi_N ( \sigma(i^\prime, j^\prime)))$, 
each with at most one choice for 
 $ \sigma( \omega_{1} (b^\prime, i^\prime)) $
 and at most
 $(N^2- 6)! $
 possible choices for the rest of values of 
 $ \sigma$. Therefore
 \begin{align}\label{sum:4}
 N^{-4} \cdot 
 \mathbb{E}\big(
 \sum_{D_4} F_{i,j, a, b} \cdot F_{i^\prime, j^\prime, a^\prime, b^\prime} 
 \big) < 
 N^{-4} \cdot N^8 \cdot \frac{N^8 \cdot (N^2-6)!}{N^2!}
 \xrightarrow[ N \rightarrow\infty] {}1.
 \end{align}
 
 The conclusion follows summing relations (\ref{sum:1}) - (\ref{sum:4}).
 
	\end{proof}

\begin{proof}[Proof of Theorem \ref{thm:3:1}]
${}$	

According to Lemma \ref{lemma:01}(i), for any subset $ S $ of $ [m]$, we have that 
$ \mathfrak{a}_{\pi, \overrightarrow{\sigma_N}}(S) \geq \mathfrak{a}_{\pi, \overrightarrow{\sigma_N}} $.
Therefore it suffices to show that there is some 
$ S \subseteq [ m ] $ and some $ \varepsilon >0$ 
such that
$ \mathfrak{a}_{\pi, \overrightarrow{\sigma_N}} (S ) < -1-\varepsilon $.

Next, we observe that if we remove from $\pi$ any block of $\pi$ of the form
$(k, k+1)$
such that
$\varepsilon(k) \neq \varepsilon(k+1)$ it does not change the value of 
$\mathcal{V}_{ \overrightarrow{\sigma_N}} ( \pi)$.
Indeed, if 
$ \pi(k) = k + 1 $
and
$ \varepsilon(k) \neq \varepsilon (k+1) $, 
that is 
$ \sigma_{k+ 1, N } = \top \circ  \sigma_{k, N }  \circ \top $, 
then the relation
$  \delta_{ \sigma_{k, N} ( i_k, i_{ k + 1})}^{ \top\circ \sigma_{k+1, N} (i_{k +1}, i_{ k+2})} = 1 $
is equivalent to
$ i_{k} = i_{ k+2} $.
Hence
\begin{align}\label{v1-1}
\mathcal{V}_{ \overrightarrow{\sigma_N}} ( \pi) = &
\frac{1}{N} \sum_{ \substack{1 \leq i_{ \nu} \leq N\\ 1 \leq \nu \leq m }} 
\frac{1}{N} 
\delta_{ \sigma_{k, N} ( i_k, i_{ k + 1})}^{ \top\circ \sigma_{k+1, N} (i_{k +1}, i_{ k+2})}
\prod_{ \substack{ (s, t) \in \pi\\ s \neq k \neq t } }
E \big(
 [ G_N^{ \sigma_{k, N}}]_{ i_s, i_{s+1}} \cdot
[ G_N^{ \sigma_{l, N}} ]_{ i_t, i_{t+1}} 
\big)\nonumber\\
=& \sum_{ i_{k+1} = 1}^N  \frac{1}{N} \mathcal{V}_{ \overrightarrow{\sigma_N}^\prime} (\pi^\prime)
 = 
 \mathcal{V}_{ \overrightarrow{\sigma_N}^\prime} (\pi^\prime)
\end{align}
where
$ \pi^\prime$
and
$ \overrightarrow{\sigma_N}^\prime $
are obtained by removing
$ (k, k+1) $
from $\pi$ and removing
$ \sigma_{k, N} $ and $ \sigma_{k + 1, N } $
from
$ \overrightarrow{\sigma_N } $.

Without loss of generality from now on we assume that $\pi$ does not contain a block of the form $(k, k+1)$
such that
$\varepsilon(k) \neq \varepsilon(k+1)$. Observe that non-crossing pairings which at the same time are alternating in $1$ and $*$  always contain such a pair. Next we shall show that for any other pairing (either crossing or non-crossing but non-alternating) we have that  $\mathcal{V}_{ \overrightarrow{\sigma_N}} ( \pi)$ vanishes asymptotically.

Suppose that $ \pi $ has a block of the form 
$ (k, k+1)$ 
with
$ \varepsilon(k) = \varepsilon (k+1) $.
Via a circular permutation of the set 
$ [ m ] $,
we can then assume that 
$ \pi (1) = 2 $. 
If 
$ m = 2 $, 
then the result follows from Lemma \ref{lemma:01:1}(i). If 
$m > 2 $, 
then from our first assumption, the restriction of 
$\pi $
to 
$ \{ 3, 4, \dots, m \} $
either has a crossing or a block of the form 
$ (s, s+1 ) $.


For the case
 $ \pi(1) = 2 $
and
   $ \pi(s) = s+1 $
for some 
 $ s > 2 $ 
 (see Figure 1 below),
 Lemma \ref{lemma:01:1}(ii) gives that for any $ \theta > 0 $, almost surely
 $ |\mathcal{A}_{ \pi, \overrightarrow{\sigma_N}}( \{1, 2\}))| = o(N^{\theta}) $,
 therefore
 $ \mathfrak{a}_{ \pi, \overrightarrow{\sigma_N}}( \{1, 2\})) <- 1 + \theta $.
 Hence, taking
 $ B = \{ 1, \dots, s-1\} $,
  part (ii) of Lemma \ref{lemma:01} gives that 
 $ \mathfrak{a}_{ \pi, \overrightarrow{\sigma_N}}( B)  < -1 + \theta $.

 \setlength{\unitlength}{.13cm}
 \begin{equation*}
 	\begin{picture}(20,13)
 		\put(-20.5,0){1}
 		
 		\put(-20,3){\circle*{2}} 
 		
 		\put(-15.5,0 ){2}
 		
 		\put(-15,3){\circle*{2}}
 		
 		\multiput(-11,3)(2, 0){5}{\circle*{.5}}

 		\put(-1,0 ){$s$}
 		
 		\put(0,3){\circle*{2}}
 		
 		
 		\put(4,0 ){$s+1$}
 		
 		\put(5,3){\circle*{2}}
 		
 		
 		\put(5,4){\line(0,1){4}}

 		\put(0, 3){\line(0,1){5}}
 		
 		\put(0, 8){\line(1,0){5}}

 		
 		\multiput(8,3)(2, 0){5}{\circle*{.5}}

 		
 		\linethickness{.55mm}

 		
 		\put(-20,4){\line(0,1){4}}

 		\put(-15, 3){\line(0,1){5}}
 		
 		\put(-15, 8){\line(-1,0){5}}

 	\end{picture}
 \end{equation*}
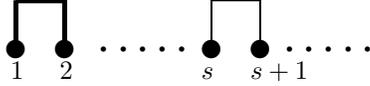
 \textbf{Figure 1}.

On the other hand, given
$ (i_\nu, j_\nu)_{ \nu \in B } \in
 \mathcal{A}_{ \pi, \overrightarrow{\sigma_N}} (B) $, 
 the component 
 $ i_s = j_{s-1}$ 
 is fixed, so according to Lemma \ref{lemma:01:1} (ii) almost surely there are 
 $ o(N^{\theta}) $
 	triples 
$( i_{s}, i_{s+1}, j_{s+1})$
such that 
$ \sigma_{ s, N}(i_s, j_s) = \top \circ \sigma_{s+1, N}(i_{s+1}, j_{s+1}) $,
that is
 \[ \mathcal{A}_{\pi, \overrightarrow{\sigma_N}} (B \cup \{ s, s+1\}) =
 o( N^{\theta} ) \cdot
  \mathcal{A}_{ \pi, \overrightarrow{\sigma_N}} (B).
  \]	
We get then 
$ \mathfrak{a}_{\pi, \overrightarrow{\sigma_N}} (B \cup \{ s , s+1\}) < -2+ 2\theta $
and the conclusion follows from Lemma \ref{lemma:01}(i).

The case 
$ \pi(1) = 2 $ 
and the set 
$ \{3, 4, \dots, m \} $ 
contains a crossing of 
$\pi $
is similar. Suppose that $ a, b, c, d $ is such a crossing
see Figure 2 below).

\setlength{\unitlength}{.13cm}
\begin{equation*}
    \begin{picture}(20,13)
\put(-20.5,0){1}

\put(-20,3){\circle*{2}} 

\put(-15.5,0 ){2}

\put(-15,3){\circle*{2}}

\multiput(-11,3)(2, 0){5}{\circle*{.5}}

\put(-1,0 ){$a$}

\put(0,3){\circle*{2}}

\multiput(3,3)(2, 0){3}{\circle*{.5}}


\put(10,0 ){$b$}

\put(10,3){\circle*{2}}

\multiput(13,3)(2, 0){3}{\circle*{.5}}


\put(20,0 ){$c$}

\put(20,3){\circle*{2}}

\multiput(23,3)(2, 0){3}{\circle*{.5}}


\put(30,0 ){$d$}

\put(30,3){\circle*{2}}

\multiput(33,3)(2, 0){5}{\circle*{.5}}


\put(20,4){\line(0,1){4}}

\put(0, 3){\line(0,1){5}}

\put(0, 8){\line(1,0){20}}


\put(30,4){\line(0,1){7}}

\put(10, 3){\line(0,1){8}}

\put(10, 11){\line(1,0){20}}




\linethickness{.55mm}


\put(-20,4){\line(0,1){4}}

\put(-15, 3){\line(0,1){5}}

\put(-15, 8){\line(-1,0){5}}
\end{picture}
\end{equation*}
\nonumber
\textbf{Figure 2}

As above, Lemma \ref{lemma:01:1}(ii) and Lemma \ref{lemma:01}(i) give that 
$ \mathfrak{a}_{\pi, \overrightarrow{\sigma_N}} ( [ b-1])
< -1 + \theta $. 
Since 
$ \pi(a) = c $,
Lemma \ref{lemma:01}(ii) gives that 
 $ \mathfrak{a}_{\pi, \overrightarrow{\sigma_N}} ( [ b-1]  \cup \{ c \})
< - 1 + \theta $,
and furthermore, utilizing again Lemma \ref{lemma:01}(i),
we get 
\[
 \mathfrak{a}_{\pi, \overrightarrow{\sigma_N}} ( [ c]  \setminus \{ b \})
< -1 + \theta.
\]
 Finally, applying part (iii) of Lemma \ref{lemma:01}, we get that  
$ \mathfrak{a}_{\pi, \overrightarrow{\sigma_N}} ( [ c] )
<  -2 + \theta $,
hence the conclusion.

Next we shall analyse the case  when
 $ \pi $ 
 does not have any block with consecutive elements and hence it has a crossing.
Let
 $ a <b < c< d$ 
be such that
 $ \pi(a) = c$,
  $ \pi(b) = d $
and 
  $ c -b $ 
is minimal. Since 
$ \pi $
does not have any block with consecutive elements, the minimality of 
$c -b $
give that 
$ c = b+1 $. 
We can also suppose, via a circular permutation of the set
$ [ m ] $,
that 
$ a =1 $.

Furthermore, remark that if the set 
$ [ d ] $ 
is not invariant under
 $ \pi $, 
 then the conclusion of the theorem holds true, that is
 $ \mathcal{V}_{ \overrightarrow{\sigma_N} }( \pi)  = o(N^{-1
}) $.
To show it, suppose that there are some 
$ t, s $
 with
 $ t < d < s $ 
 and
 $ \pi(t) = s $.
 If 
 $ b+ 1 <  t <  d $,
 (see Figure 3 below)
  then  
  $ \mathfrak{a}_{ \pi, \overrightarrow{\sigma_n}} ( \{ 1, b+1\}) = 0$,
  so applying parts (i) and (iii) of Lemma \ref{lemma:01} 
  we get that 
  $\mathfrak{a}_{ \pi, \overrightarrow{\sigma_n}} ([ b+1]) \leq -1 $.

\setlength{\unitlength}{.13cm}
\begin{equation*}
	\begin{picture}(20,13)
		\put(-20.5,0){1}
		
		\put(-20,3){\circle*{2}} 
		
		
		
		\multiput(-17,3)(2, 0){5}{\circle*{.5}}

		\put(-7,0 ){$b$}
		
		\put(-6,3){\circle*{2}}
		
		\put(-3,0 ){$b+1$}
		
		\put(-1,3){\circle*{2}}
		
		\multiput(3,3)(2, 0){3}{\circle*{.5}}
		
		\put(10,0 ){$t$}
		
		\put(10,3){\circle*{2}}
		
		\multiput(13,3)(2, 0){3}{\circle*{.5}}
		
		
		\put(20,0 ){$d$}
		
		\put(20,3){\circle*{2}}
		
		\multiput(23,3)(2, 0){3}{\circle*{.5}}
		
		
		\put(30,0 ){$s$}
		
		\put(30,3){\circle*{2}}
		
		\multiput(33,3)(2, 0){5}{\circle*{.5}}

		
		\put(30,4){\line(0,1){7}}

		\put(10, 3){\line(0,1){8}}
		
		\put(10, 11){\line(1,0){20}}

		

		
		\linethickness{.55mm}
		
		
		\put(20,4){\line(0,1){4}}

		\put(-1, 3){\line(0,1){8}}
		
		\put(-6, 8){\line(1,0){26}}
		
		\put(-6, 3){\line(0,1){5}}
		
		\put(-20,3){\line(0,1){8}}
		
		\put(-20,11){\line(1,0){19}}
		
		

	\end{picture}
\end{equation*}
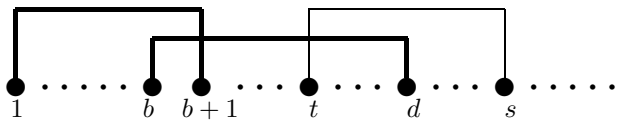
\textbf{Figure 3}. 
  
On the other hand, part (ii) of Lemma \ref{lemma:01} gives then that 
 $\mathfrak{a}_{ \pi, \overrightarrow{\sigma_n}} ([ b+1])
 \geq
\mathfrak{a}_{ \pi, \overrightarrow{\sigma_n}} ([ b+1] \cup \{ d \} ) $
and another application of parts (i) and (iii) give that
\[
\mathfrak{a}_{ \pi, \overrightarrow{\sigma_n}} ([ d ])
\leq 
\mathfrak{a}_{ \pi, \overrightarrow{\sigma_n}} ([ d ] \setminus \{ t \} ) - 1 
\leq 
\mathfrak{a}_{ \pi, \overrightarrow{\sigma_n}} ([ b+1] \cup \{ d \} ) - 1 \leq -2,
\]
hence the conclusion. The argument for the case 
$ 1 < t < b $ 
is similar.

 We can assume then that 
$ \pi ( [ d ]) = [d] $.
 Furthermore, this allows us to assume that 
$ d = m $:
 if 
$ d < m $, 
then the restriction of 
$ \pi $
to 
$ [m ]\setminus [ d ] $ 
has some crossing (since no block has consecutive elements), and again Lemma \ref{lemma:01} gives that
$ \mathfrak{a}_{ \pi \overrightarrow{\sigma_N}} \leq -2 $.
Finally, without loss of generality, we can also suppose that 
$ \varepsilon(1) =1 $,
that is 
$ \sigma_{1, N} = \sigma_N $.

If
$ m = 4 $,
then, putting
$\omega_{s, N} \in \{ \id_N, \top \} $
in part (i) of Lemma \ref{lemma:1:2:3}   we obtain that 
$ \mathcal{A}_{\pi, \overrightarrow{\sigma_N}} = o(N^{1+\theta})$,  
hence
\[
\mathfrak{a}_{ \pi, \overrightarrow{\sigma_N}} < 
 1+ \theta - | \{ 1, 2, 3, 4\} | + | \{ (1, 3), (2, 4)\} |  -1 = -2 + \theta.
\]

If 
$ m > 4 $, 
then applying part (ii) of Lemma \ref{lemma:1:2:3}
for $(i, j, a, b) = (i_1, i_b, i_{b+1}, i_m)$, 
$\omega_{s, N} \in \{ \id_N, \top \} $
and 
$ \phi_N $ one of the canonical projections,
   we obtain that 
\[
 \mathcal{A}_{ \pi, \overrightarrow{\sigma_N}}
  ( \{ 1, b, b+1, m \}) = o(N^{\frac{5}{2} +\theta}),\]
hence
\[
\mathfrak{a}_{ \pi, \overrightarrow{\sigma_N}}
( \{ 1, b, b+1, m \})
<
\frac{5}{2}+ \theta - 4 + 2 - 1 = -\frac{1}{2} + \theta.
\]

On the other hand, since
$ m > 4 $ 
and 
$ \pi $
does not have any blocks with consecutive elements, 
then either the sets 
$ \{ 2, 3 \dots, b-1  \} $ 
and
$ \{ b+2, b+3, \dots, m-1 \} $
are invariant under 
$ \pi $ 
and at least one of them is nonvoid,
or there are some 
$ t, s $ 
 with 
 $ \pi (t) = s $
 and
$ 1 < t < b $
and
$ b+1 < s < m $.

Let us suppose that the set
 $\{2, 3, \dots, b-1\}$
 is nonvoid and invariant under 
 $\pi $.
 Since 
 $\pi $ 
 does not have any blocks with consecutive elements, the restriction of $ \pi $
  to the set
 $\{2, 3, \dots, b-1\}$
 contains at least one crossing, 
 $ ( a^\prime, b^\prime, c^\prime, d^\prime) $ 
 (see Figure 4 below).
 

 \setlength{\unitlength}{.13cm}
 \begin{equation*}
 \begin{picture}(20,13)
 \put(-20.5,0){1}
 
 \put(-20,3){\circle*{2}} 
 
 
 
 \multiput(-17,3)(2, 0){3}{\circle*{.5}}

 \put(-11,0 ){$a^\prime$}
 
 \put(-10,3){\circle*{2}}
 
 \multiput(-7,3)(2, 0){3}{\circle*{.5}}
 
 
 \put(0,0 ){$b^\prime$}
 
 \put(0,3){\circle*{2}}
 
 \multiput(3,3)(2, 0){3}{\circle*{.5}}
 
 
 \put(10,0 ){$c^\prime$}
 
 \put(10,3){\circle*{2}}
 
 \multiput(13,3)(2, 0){3}{\circle*{.5}}
 
 
 \put(20,0 ){$d^\prime$}
 
 \put(20,3){\circle*{2}}
 
 \multiput(23,3)(2, 0){3}{\circle*{.5}}
 
 
 \put(10,4){\line(0,1){4}}

 \put(-10, 3){\line(0,1){5}}
 
 \put(-10, 8){\line(1,0){20}}

 
 \put(20,4){\line(0,1){7}}

 \put(0, 3){\line(0,1){8}}
 
 \put(0, 11){\line(1,0){20}}
 
 
 \put(29,0 ){$b$}
 
 \put(30,3){\circle*{2}}
 
 
 \put(34,0 ){$b+1$}
 
 \put(35,3){\circle*{2}}
 
 \multiput(38,3)(2, 0){5}{\circle*{.5}}
 
 
 \put(48,0 ){$m$}
 
 \put(48,3){\circle*{2}}
 
 
 

 
 \linethickness{.55mm}

 
 \put(-20,3){\line(0,1){10}}
 
 \put(35,3){\line(0,1){10}}
 
 \put(-20,13){\line(1,0){55}}
 
 
 \put(30,3){\line(0,1){7}}
 
 \put(48,3){\line(0,1){7}}
 
 \put(30,10){\line(1,0){18}}

 

 \end{picture}
 \end{equation*}
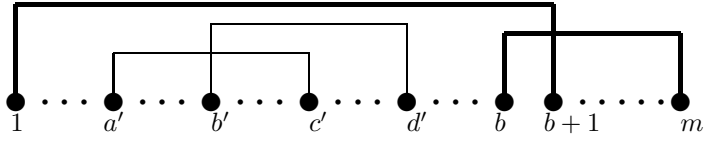
 \textbf{Figure 4}. 
 

 Applying parts (i), then (ii) of Lemma \ref{lemma:01}, we have that:
 \begin{align*}
 \mathfrak{a}_{ \pi, \overrightarrow{\sigma_N}}
 ( \{ 1, b, b+1, m\}) 
 \geq
 \mathfrak{a}_{ \pi, \overrightarrow{\sigma_N}}
 ( [ a^\prime] \cup \{ b, b+1, m \})
\geq
 \mathfrak{a}_{ \pi, \overrightarrow{\sigma_N}}
 ( [ a^\prime] \cup \{c^\prime, b, b+1, m \}).
 \end{align*}
 Denoting 
 $ B = [ c^\prime] \cup \{ b, b+1, m\} $ 
 and applying again Lemma \ref{lemma:01}(i), we have further obtain
 \begin{align*}
\mathfrak{a}_{ \pi, \overrightarrow{\sigma_N}} ( B \setminus \{ b^\prime\})
\leq 
 \mathfrak{a}_{ \pi, \overrightarrow{\sigma_N}}
( [ a^\prime] \cup \{c^\prime, b, b+1, m \}),
 \end{align*}
 and another application of Lemma \ref{lemma:01}(iii)
 gives that
 \begin{align*}
 \mathfrak{a}_{ \pi, \overrightarrow{\sigma_N}} ( B )
 \leq
 \mathfrak{a}_{ \pi, \overrightarrow{\sigma_N}} 
 ( B \setminus \{ b^\prime\}) -1 \leq 
 \mathfrak{a}_{ \pi, \overrightarrow{\sigma_N}}
 ( \{ 1, b, b+1, m\}) 
 -1 < -\frac{3}{2}+ \theta,
 \end{align*}
 hence the conclusion. The case when 
 $ \{ b+2, b+2, \dots, m-1\} $
 is nonvoid is similar.
 
 Finally, let us suppose that there are some 
 $ t, s $ 
 with 
 $ \pi (t) = s $
 and
 $ 1 < t < b $
 and
 $ b+1 < s < m $
 (see Figure 5 below).

 \setlength{\unitlength}{.13cm}
 \begin{equation*}
 \begin{picture}(20,13)
 \put(-20.5,0){1}
 
 \put(-20,3){\circle*{2}} 
 
 \multiput(-17,3)(2, 0){3}{\circle*{.5}}
 
 \put(-10,0 ){$t$}
 
 \put(-10,3){\circle*{2}}
 
 \multiput(-7,3)(2, 0){3}{\circle*{.5}}
 
 
 \put(-1,0 ){$b$}
 
 \put( 0,3){\circle*{2}}
 
 \put(3,0 ){$b+1$}
 
 \put( 5,3){\circle*{2}}
 
 \multiput(8,3)(2, 0){3}{\circle*{.5}}

 
 
 \put(15,0 ){$s$}
 
 \put(15,3){\circle*{2}}
 
 \multiput(18,3)(2, 0){3}{\circle*{.5}}
 
 
 \put(24,0 ){$m$}
 
 \put(25,3){\circle*{2}}
 
 
 \put(-10,4){\line(0,1){9}}

 \put(15, 3){\line(0,1){10}}
 
 \put(-10, 13){\line(1,0){25}}
 
 
 \linethickness{.55mm}
 
 
 \put(-20,4){\line(0,1){3}}
 
 \put( 5,3){\line(0,1){4}}
 
 \put(-20,7){\line(1,0){25}}
 
 
 \put( 0,3){\line(0,1){6}}
 
 \put(25,3){\line(0,1){6}}
 
 \put( 0,9){\line(1,0){25}}
 
 \end{picture}
 \end{equation*}
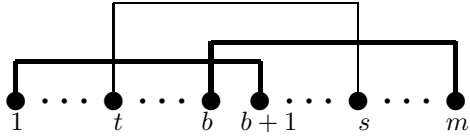
 \textbf{Figure 5}. 
 

 Applying part(i), then part (iii) of the Lemma \ref{lemma:01} we obtain that
 \begin{align*}
-\frac{1}{2} + \theta >
 \mathfrak{a}_{ \pi, \overrightarrow{\sigma_N}}
 ( \{ 1, b, b+1, m\}) 
 \geq
 \mathfrak{a}_{ \pi, \overrightarrow{\sigma_N}}
 ( [ b+1] \setminus \{ t \} \cup\{ m \} )
 \end{align*}
respectively that
\begin{align*} \mathfrak{a}_{ \pi, \overrightarrow{\sigma_N}}
( [ b+1] \setminus \{ t \} \cup\{ m \} )
\geq
\mathfrak{a}_{ \pi, \overrightarrow{\sigma_N}}
( [ b+1]  \cup\{ m \} ) + 1,
\end{align*}
 hence the conclusion.

\end{proof}


\section{Asymptotic Infinitesimal Freeness Relations}
 In \cite{mvp_gaussian}, it is shown that $G_N^{\sigma_N}$ and $G_N^{\tau_N}$ are asymptotically circular and asymptotically free whenever the pair of permutation sequences $(\sigma_N)_N$ and $(\tau_N)_N$ satisfies certain conditions. Similar properties for Haar matrices were studied in \cite{mpsz1} where we also allowed that the sequences of permutations are random. One of the consequences of \cite{mpsz1} (not stated explicitly anywhere) is that  conditions from \cite{mvp_gaussian} hold almost surely, which implies that for almost all pairs of permutation sequences $(\sigma_N)_N$ and $(\tau_N)_N$, we have $G_N^{\sigma_N}$ and $G_N^{\tau_N}$ are asymptotically circular and asymptotically free.

In this section, we show that the property also holds on the level of infinitesimal distributions. More precisely, we show the following theorem: 

\begin{thm}\label{main:2}
 Suppose that 
 $ G_N $ 
 is a Gaussian random matrix for each positive integer 
 $ N $.
 Moreover assume that sequences of uniformly random permutations 
 $ \big( \sigma_N \big)_N $  and $ \big( \tau_N \big)_N $ from  $ \mathcal{S}( [ N]^2) $  are independent from each other and independent from the sequence $(G_N)_N$.
 We have almost surely that 
  $ G_N^{\sigma_N} $ and $ G_N^{ \tau_N } $
  are asymptotically circular with zero infinitesimal distribution 
  and infinitesimally free from each other and from 
 $ G_N $.
\end{thm}

\begin{remark}
 Let us clarify what we mean by the ''almost sure'' statement above. We consider asymptotic freeness with respect to sequence of functionals
  $(1/N) \mathbb{E}\circ \Tr$ 
  i.e. we take the expectation of the normalized trace. We will study as in the previous section the quantity $ \mathcal{V}_{ \overrightarrow{\sigma_N}} ( \pi)$ defined as in equation \eqref{eq:1}, which becomes a random variable when we let $\overrightarrow{\sigma_N}$ be a sequence of random permutations. Hence the almost sure limit here refers to the limit with probability one with respect to an uniformly random permutation.
\end{remark}

The proof of the theorem above is very similar to the argument from the previous section. More precisely, in equation (\ref{eq:1}) we consider that each 
$ \sigma_{ k, N } $ 
is one of the following permutations:
$ \id_N,  \sigma_N, \top \circ\sigma_N\circ \top , \tau_N, \top \circ \tau_N \circ \top $
and prove a results similar to the ones from Section \ref{Section:3}.

\begin{lemma}\label{lemma:02:1}
Suppose that 
$ \big( \omega_N \big)_N $
is a  given sequence of permutations with 
$ \omega_N \in \mathcal{S}( [ N ] ^2 ) $ 
for each 
$ N $. 
If $\theta > 0 $ is a constant, then 
for almost all
$ ( \sigma_N)_N $
we have that:
\begin{enumerate}
\item[(i)]
$  \displaystyle \lim_{ N \rightarrow \infty }
N^{-(\frac{1}{2} + \theta)} \cdot 	| \{ (i, j) \in [ N ]^2 : \ \omega_N( i, j) = \sigma_N (j, i)  \}  | 
= 0 $
\item[(ii)] 
$ \displaystyle 
\lim_{N \rightarrow \infty}
N^{-\theta} \cdot
\sup_{ 1 \leq i \leq N } \big|  (j, k) \in [ N ]^2:\ 
\omega_N ( i, j) =  \sigma_N (j, k) \}  \big| = 0.
$
\end{enumerate}
\end{lemma}

\begin{proof}
 For part (i), given 
 $ i, j \in [ N ] $,
 denote by 	
$ I_{ i, j} $
the random variable on 
$ \mathcal{S} ( [N]^2)	$
given by
$ I_{i, j} ( \sigma) = \left\{
\begin{array}{ll}
1 & \textrm{ if } \sigma (i, j)= \omega_N (j, i) \\
0 & \textrm{ if } \sigma (i, j)\neq \omega_N (j, i).
\end{array}
\right. $

As in the proof of Lemma \ref{lemma:01:1}(i), it suffices to show that the sums 
$ \displaystyle \sum_{i, j=1}^N \mathbb{E}(I_{i, j}) $
and
$ \displaystyle \sum_{\substack{ 1 \leq i, j, k, l \leq N \\(i, j) \notin \{ (k, l), (l, k)\} } } 
\mathbb{E} ( I_{i, j} I_{k, l}) $
are bounded in
 $ N $.
 
 For 
 $ I_{i, j} ( \sigma) \neq 0 $
 there is only one possible choice for 
 $ \sigma_N(j, i) $ 
 and 
 $ ( N^2 -1) ! $ 
 possiblilities for the rest of the values of 
 $ \sigma $, 
 therefore
\begin{align*}
\sum_{i, j=1}^N \mathbb{E}(I_{i, j}) \leq N^2 \frac{1 \cdot (N^2 -1)! }{N^2!}  = 1.
\end{align*}

If
$(i,j) \notin \{ (k, l), (l, k) \} $,
for 
$ I_{i, j} (\sigma) I_{k, l}(\sigma) \neq 0 $
there is only one possible choice for
$ \sigma (j, i) $ 
and for 
$ \sigma ( l, k) $
and 
$ ( N^2 -2)! $
possible choices for the rest of the values of 
$ \sigma $,
therefore
\begin{align*}
\sum_{\substack{ 1 \leq i, j, k, l \leq N \\(i, j) \notin \{ (k, l), (l, k)\} } } 
\mathbb{E} ( I_{i, j} I_{k, l}) 
\leq N^2 (N^2 -2) \frac{1\cdot (N^2 - 2)! }{N^2!}
\xrightarrow[{N \rightarrow \infty}]{} 1 ,
\end{align*}
and the conclusion follows.

For part (ii), for a fixed 
$ a \in [ N ] $
 denote by 
 $ F_{i, j} $ the random variable on 
 $ \mathcal{S}([N]^2) $
 given by
 $ F_{i, j} ( \sigma) = 
  \left\{
 \begin{array}{ll}
 1 & \textrm{ if } \sigma (i, j)= \omega_N (a, i) \\
 0 & \textrm{ if } \sigma (i, j)\neq \omega_N (a, i)
 \end{array}
 \right.
 $
 and, as in the proof of Lemma \ref{lemma:01:1}(ii), it suffices to show the boundedness (as 
 $ N \rightarrow \infty $)
 of the expression
 $$
 \mathbb{E} \big(
 \sum_D F_{i_1, j_1} \cdots F_{i_s, j_s} \big) $$
 where 
 $ s $
 is some positive integer and again
 \begin{align*}
 D = \{ (i_1, j_1, \dots, i_s, j_s)\in [ N ]^{2s} :\ (i_1, j_1), \dots, (i_s, j_s) \textrm{ distinct } \}.
 \end{align*} 
 
If
 $ F_{i_t, j_t}  \neq 0 $
 for each 
 $ 1 \leq t \leq s $, 
 we need that the $ s$-tuple 
 $ ( \sigma(i_1, j_1), \dots, \sigma(i_s, j_s)) $
 coincides to the $ s $-tuple
 $ ( \omega_N (a, i_1), \dots, \omega_N (a, i_s) )$
 and there are at most 
 $ (N^2 -s)! $ 
 choices for the rest of the values of
  $ \sigma$.
   Therefore
   \begin{align*}
 \mathbb{E} \big(
 \sum_D F_{j_1, k_1} \cdots F_{i_s, j_s} \big)
 \leq 
 \frac{1}{N^2!} \cdot N^{2s} \cdot (N^2 -s)!
 \xrightarrow[{N \rightarrow \infty}]{}1 
   \end{align*}
   hence the conclusion.
\end{proof}
	
\begin{lemma}\label{lemma:02:2}
	Let $ \theta$ be a constant.
The following relations hold true almost surely
for sequences 
$ \big( \sigma_N \big)_N $ 
with each
$ \sigma_N $ 
an uniformly random permutation from $ \mathcal{S}([N]^2) $:
\begin{enumerate}
\item[(i)] if  $ (f_N)_N , (g_N)_N $ 
are two given sequences of maps, $ f_N, g_N : [ N]^2 \rightarrow [ N]^2 $, then
 \[ \displaystyle 
\lim_{N \rightarrow \infty }
 N^{-(1+ \theta)} \cdot
\big|
\{ (i, j):\ 
\sigma_N ( f_N(i, j)) = g_N (i, j)
\}\big| = 0 \]
\item[(ii)] if $ ( \phi_N)_N , ( \psi_N)_N $ 
are sequences of maps, $ \phi_N, \psi_N :[ N]^3 \rightarrow [N]^2$,
such that 
$(i, a, b) = (i^\prime, a^\prime, b^\prime)$
whenever
$ \phi_N (i, j, a)= \phi_N(i^\prime, j^\prime, a^\prime)$
and
$\psi_N(i, j, b) = \psi_N(i^\prime, j^\prime, b^\prime)$,
 then
\[\displaystyle \lim_{ n \rightarrow \infty}
N^{-(\frac{5}{2} + \theta)} \cdot \big|\{ (i, j, a, b) : \ 
 \sigma_N ( \phi_N(i, j, a)) = \psi_N (i, j, b)
\}
\big|
= 0.
\]
\end{enumerate}	
\end{lemma}	

\begin{proof}
Considering the random variable
$ I_{i, j}( \sigma) =
\left\{ 
\begin{array}{ll}
1 & \textrm{ if }  \sigma ( f_N (i, j)) = g_N (i, j)\\
0 & \textrm { otherwise }
\end{array}
\right.$ 
and using Markov's inequality, we have that
\begin{align*}
\mathbb{P} \big( \big|\{ (i, j):\ 
\sigma_N ( ( f_N (i, j)) = g_N (i, j)
\}\big|
> \varepsilon \cdot N^{1+\theta} 
\big) 
< \frac{ \mathbb{E} \big(   \displaystyle \sum_{i, j \in [ N ] } I_{i, j}\big) }{ \varepsilon \cdot N^{1+\theta} }
\end{align*}
so by Lemma \ref{lem:BC} it suffices to show that 
$ \displaystyle \mathbb{E} \big( \sum_{i, j \in [ N ] } I_{i, j} \big) = O(N^0) $.

For 
$ I_{i, j}( \sigma) = 1 $,
there is at most one possible choice for 
$ \sigma( f_N (i, j)) $ 
and at most
$ ( N^2 -1 )! $
possible choices for the rest of the values of 
$ \sigma $,
hence
\begin{align*}
\mathbb{E} \big( \sum_{i, j \in [ N ] } I_{i, j} \big) 
\leq N^2 \cdot \frac{1 \cdot (N^2 -1 ) ! }{N^2!}
\xrightarrow[ N \rightarrow\infty] {} 1.
\end{align*}

Similarly, for part (ii) considering the random variable
\[
I_{i, j,a, b} (\sigma) = \left\{ 
\begin{array}{ll}
1 & \textrm{ if } \sigma (\phi_N (i, j, a)) = \psi_N (i, j, b)\\
0 & \textrm{ otherwise }
\end{array}
\right.
\]
and applying again Markov's inequality,  we have that
\begin{align*}
\mathbb{P} \big( 
\big| \{ (i, j, a, b) :\sigma_N ( \phi_N (i, j, a)) = \psi_N (i, j, b) \}\big| 
> 
\varepsilon \cdot N^{\frac{5}{2} + \theta} 
\big)< \frac{ \displaystyle \mathbb{E} \big(
(\sum_{i, j, a, b} I_{i, j, a, b} )^2\big)  }{\varepsilon^2 \cdot N^{5+ 2\theta}}
\end{align*}
so by Lemma \ref{lem:BC} it suffices to show that 
$  \displaystyle \mathbb{E} \big(
(\sum_{i, j, a, b} I_{i, j, a, b})^2\big) = O(N^{4})
$.

For 
$ I_{i,j,a, b} (\sigma) =1 $
there is only one possible choice for 
$ \sigma( \phi_N(i, j, a))$
and
$ (N^2 -1)!$ 
possible choices for the other values of 
$ \sigma$,
therefore
\begin{align*}
\frac{1}{N^4}\mathbb{E} \big( \sum_{i, j,a, b} I_{i, j, a, b}^2 \big)= 
\frac{1}{N^4}\mathbb{E} \big( \sum_{i, j,a, b} I_{i, j, a, b} \big)
\leq\frac{1}{N^4} \cdot  N^4 \cdot \frac{ (N^2-1)!}{N^2}
\xrightarrow[N \rightarrow \infty]{}0.
\end{align*}
Denote
$D =\{ (i, j,a, b, i^\prime, j^\prime, a^\prime, b^\prime):\ (i, a, b) \neq (i^\prime, a^\prime, b^\prime) \}$.
From the property of 
$\phi_N$
 and 
$\psi_N $,
  it follows that if 
$(i, j,a, b, i^\prime, j^\prime, a^\prime, b^\prime)\in D $
and
$ \phi_N(i, j, a) = \phi_N(i^\prime, j^\prime, a^\prime) $
then
$ I_{i, j, a, b}I_{i^\prime, j^\prime, a^\prime, b^\prime}
= 0 $.
Hence, if
$(i, j,a, b, i^\prime, j^\prime, a^\prime, b^\prime)\in D $
and the relation
$ I_{i, j, a, b}\cdot I_{i^\prime, j^\prime, a^\prime, b^\prime}
(\sigma)
\neq 0 $
holds true, then
there is one possible choice for
$ \sigma( \phi_N(i, j, a)) $ and
$ \sigma( \phi_N(i^\prime, j^\prime, a^\prime)) $
and
$ (N^2 -2)! $
possible choices for the rest of the values of 
$ \sigma $. 
Therefore
\begin{align*}
\frac{1}{N^4} \mathbb{E}\big( \sum_D I_{i,j, a, b}I_{i^\prime, j^\prime, a^\prime, b^\prime} \big) 
\leq \frac{1}{N^4} \cdot N^8 \cdot \frac{(N^2 -2)!}{N^2!}
\xrightarrow[N\rightarrow\infty]{} 1.
\end{align*}

On the other hand, for 
 $ I_{i, j, a, b}\cdot I_{i, j^\prime, a, b}
(\sigma)
\neq 0 $, 
there is one possible choice for 
$ \sigma( \phi_N(i, j, a)) $
and at most
$ (N^2 -1) ! $
possible choices for the other values of 
$ \sigma$, hence
\begin{align*}
\frac{1}{N^4} \mathbb{E} \big(
 \sum_{i, a, b, j, j^\prime} 
 I_{i, j, a, b}\cdot I_{i, j^\prime, a, b}  
  \big) \leq \frac{1}{N^4}\cdot N^5 \cdot \frac{(N^2-1)!}{N^2!}
  \xrightarrow[ N \rightarrow \infty ]{} 0,
\end{align*}
and property (ii) follows.
\end{proof}	

\begin{lemma}\label{lemma:02:3}
The following relations hold true almost surely
for sequences 
$ \big( \sigma_N \big)_N $ 
with each
$ \sigma_N $ 
a permutation from $ \mathcal{S}([N]^2) $:
\begin{enumerate}
\item[(i)] Suppose that
$ ( f_N)_N $, $ ( g_N)_N $ 
are two given sequences of maps 
and 
$ ( \eta_N)_N $
is a sequence of permutations 
such that 
$ f_N, g_N : [ N]^2 \rightarrow [N]^2 $
and 
$ \eta_N \in \mathcal{S}([ N ]^2).$
%
Then
\begin{align*}
\lim_{N \rightarrow \infty} 
N^{-(1+\theta)} \cdot 
\big| 
\{ (i, j) \in [N]^2:\ f_N(i, &j)  \neq g_N(i, j)\\
  \textrm{ and } 
 &  \sigma_N ( f_N(i, j)) = \eta_N \circ \sigma_N ( g_N(i, j))  \}
\big| = 0.
\end{align*}
\item[(ii)] Suppose that 
$ (h_N)_N $ 
is a sequence of maps and
 $ (\omega_N)_N $, $ (\eta_N)_N $
are sequences of permutations,
$ h_N : [N]^2 \rightarrow [N] $ 
and 
$ \omega_N, \eta_N \in \mathcal{S}([ N]^2) $
with
$ (\omega_N)_N $ 
either the identity permutation or the matrix transpose.
Then
\begin{align*}
\lim_{N \rightarrow \infty }
 N^{-(\frac{5}{2} + \theta)} \cdot  \big| \{ (i, j, a, b) \in [ N ]^4:\ 
  &\omega_N ( b, i) \neq ( a, h_N(i, j)) \textrm{ and }\\
 &\sigma_N ( a, h_N(i, j) ) = \eta_N \circ \sigma_N \circ \omega_N ( b, i)\}
 \big| = 0
\end{align*}
\end{enumerate}
\end{lemma}

\begin{proof}
For part (i), denote 
$ D = \{ (i, j) \in [ N ]^2:\  f_N(i, j) \neq g_N(i, j) \}$
and consider
the random variable on 
$ \mathcal{S}([ N]^2) $
given by
\[ I_{i, j} (\sigma) = \left\{ 
\begin{array}{ll}
1 & \textrm{if }  \sigma ( f_N(i, j)) = \eta_N \circ \sigma ( g_N(i, j)) \\
0 & \textrm{otherwise}.
\end{array}
\right. \]
	 

Via Markov's inequality and Lemma \ref{lem:BC} it suffices to show that 
$ \displaystyle \mathbb{E} 
\big( \sum_{(i,j) \in D } I_{i, j}  \big) = O(N^0)
$

For 
$ I_{i, j} ( \sigma) = 1 $, 
there are 
$ N^2$ 
possible for 
$ \sigma_N ( f _N(i, j)) $,
each giving one possible choice for 
$ \sigma_N ( g_N(i, j)) $ 
and 
$ (N^2 -2)! $
possible choices for the rest of values of 
$ \sigma $. 
Therefore:
\begin{align*}
\mathbb{E}\big( \sum_{(i,j) \in D } I_{i, j}  \big)
< N^2 \cdot \frac{N^2 \cdot (N^2 -2)!}{N^2!} 
\xrightarrow[ N \rightarrow \infty ]{}1,
\end{align*}
hence the conclusion.

For part (ii), denote  
$ V =\{ (i, j, a, b) \in [ N ]^4 :\ \omega_N ( b, i)\neq (a, h_N(i, j)) \}$
and consider the random variable on 
$ \mathcal{S}([ N]^2) $
given by
\begin{align*}
F_{i, j, a, b} (\sigma) = \left\{ 
\begin{array}{ll}
1 & \textrm{ if }  \sigma (a, h_N(i, j)) = \eta_N \circ  \sigma \circ \omega_N (b, i)\\
0 & \textrm{ otherwise. }
\end{array}
 \right. 
\end{align*}
With these notations, it suffices to show that 
$ \displaystyle  \mathbb{E}
 \big( ( \sum_V  F_{i, j, a, b} )^2 \big) 
 = O(N^{4}) $.
 
To simplify the writing, we shall use the notations
$ \vec{v} $ 
(respectively
$ \vec{v}^\prime $)
 for 
 $ (i, j, a, b)$
 (respectively
 $( i^\prime, j^\prime, a^\prime, b^\prime)$);
 also, let
 $ Z(\vec{v}) = \{ (a, h_N(i, j)),  \omega_N ( b, i)\}$
 and introduce the sets
 $ W = \{ (\vec{v}, \vec{v}^\prime ) \in V \times V :\ 
  \vec{v} \neq \vec{v}^\prime \} $,
  and, for
  $ s =0, 1, 2 $,
  let 
  \begin{align*}
  W_s = \{ (\vec{v}, \vec{v}^\prime ) \in W :\  \big|
  Z( \vec{v}) \cap Z(\vec{v}^\prime)) \big| = s             \}.
  \end{align*} 
  We have that 
  \begin{align*}
 \mathbb{E}
 \big( ( \sum_V  F_{i, j, a, b} )^2 \big) & =
  \mathbb{E}
 \big( \sum_V  F_{i, j, a, b}^2 \big)
 + 
 \mathbb{E} \big( 
 \sum_{W} F_{i, j, a, b} F_{i^\prime, j^\prime, a^\prime, b^\prime}
 \big)\\ 
& = 
   \mathbb{E}
 \big( \sum_V  F_{i, j, a, b} \big)
 + \sum_{s=0}^2 
  \mathbb{E} \big( 
 \sum_{W_s} F_{i, j, a, b} F_{i^\prime, j^\prime, a^\prime, b^\prime}
 \big).
  \end{align*}
Since  the random variables 
$ F_{i, j, a, b} $
are identically distributed,  
\begin{align*}
\frac{1}{N^4}  \mathbb{E}
\big( \sum_V  F_{i, j, a, b} \big) =
 \big| V \big| \cdot \frac{1}{N^4} \mathbb{E}
 \big(   F_{i, j, a, b} \big) < N^4 \cdot \frac{1}{N^4} = 1.
\end{align*}  

 If 
 $ (\vec{v}, \vec{v}^\prime) \in W_0 $,
 then there are at most 
 $ N^2 $ 
 possible choices for 
 $  \sigma( a, h_N(i, j))$,
 each giving one possible choice for 
 $ \sigma ( \omega_N(b, i))\big) $, 
  less than
 $ N^2 $
 possible choices for the pair 
 $\big( \sigma( a^\prime, h_N(i^\prime, j^\prime)) , \sigma ( \omega_N(b^\prime, i^\prime)) \big)$ 
 and at most 
 $ (N^2 - 4)! $ 
 possible choices for the rest of the values of 
 $ \sigma $. Therefore:
 \begin{align*}
 \frac{1}{N^4} \cdot  \mathbb{E} \big( 
 \sum_{W_0} F_{i, j, a, b} F_{i^\prime, j^\prime, a^\prime, b^\prime}
 \big)
 \leq
  \frac{1}{N^4} \cdot N^8 \cdot \frac{N^4 \cdot (N^2 -4)!}{N^2!}
 \xrightarrow[ N \rightarrow \infty]{}1.
 \end{align*}
 
 If
 $ (\vec{v}, \vec{v}^\prime) \in W_1 $
 then there are at most 
 $ N^2$ 
 possible choices for 
 $ \sigma( a, h_N(i, j)) $
 each giving one choice for 
 $ \sigma ( Z(\vec{v}) \cup Z(\vec{v}^\prime) $
 and at most 
 $ (N^2 -3)! $ 
 possible choices for the rest of values of 
 $ \sigma $. 
 Also, since 
 $ \omega_N $ 
 is either the identity of the matrix transpose, we have that 
 $ | W_1 | \leq N^7 $. 
 Therefore
 \begin{align*}
  \frac{1}{N^4} \cdot  \mathbb{E} \big( 
 \sum_{W_1} F_{i, j, a, b} F_{i^\prime, j^\prime, a^\prime, b^\prime}
 \big)
 \leq
 \frac{1}{N^4} \cdot N^7 \cdot \frac{N^2 \cdot (N^2 -3)!}{N^2!}
 \xrightarrow[ N \rightarrow \infty]{}0.
 \end{align*}
 
 If
 $ (\vec{v}, \vec{v}^\prime) \in W_2 $
 then there are at most 
 $ N^2$ 
 possible choices for 
 $ \sigma( a, h_N(i, j)) $
  each giving one choice for 
 $ \sigma ( Z(\vec{v}) \cup Z(\vec{v}^\prime) $
 and at most 
 $ (N^2 -2)! $ 
 possible choices for the rest of values of 
 $ \sigma $.
 Moreover, if
 $ (a, h_N(i, j)) = (a^\prime, h_N(i^\prime, j^\prime)) $
 and 
 $\omega_N(b, i) = \omega_N(b^\prime, i^\prime)$, 
 then
 $ (a, b, i) = (a^\prime, b^\prime, i^\prime)$, 
 and
 $ | W_2 | \leq N^5$.
  If 
  $ (a, h_N(i, j)) = \omega_N(b^\prime, i^\prime) $
  and 
  $ \omega_N(b, i) = (a^\prime, h_N(i^\prime, j^\prime)) $,
  then
  $a=b^\prime$ and $b=a^\prime$ 
  when 
  $\omega_N$ 
  is the identity, respectively
  $a=i^\prime$ and $i=a^\prime$
  when
  $\omega_N $ is the matrix transpose.
  Hence 
  $ | W_2 | \leq N^6$.
  We have that
  \begin{align*}
 \frac{1}{N^4} \cdot  \mathbb{E} \big( 
 \sum_{W_2} F_{i, j, a, b} F_{i^\prime, j^\prime, a^\prime, b^\prime}
 \big)
 \leq
 \frac{1}{N^4} \cdot N^6 \cdot \frac{N^2 \cdot (N^2 -2)!}{N^2!}
 \xrightarrow[ N \rightarrow \infty]{}1
 \end{align*}
and the conclusion follows.
\end{proof}	
\begin{lemma}\label{lemma:02:4}
Suppose that 
$ (\eta_{1, N})_N, (\eta_{2, N})_N, (\omega_N)_N $ 
are given sequences of permutations such that
$ \omega_N $ 
is either the identity or the matrix transpose.
Then, almost surely for 
$ ( \sigma_N)_N $ 
we have that
\begin{align*}
\lim_{N \rightarrow \infty} N^{-(\frac{5}{2} + \theta)} \cdot 
\big| \{ (i, j, k, l, a, b) \in [ N ]^6:\ (a, k) \neq& \omega_N (k,  l) \textrm{ and } 
\sigma_N (a, k) = \eta_1(b, i), \\
 & \sigma_N ( \omega_N (k, l)) = \eta_2(i, j) \}
\big| = 0.
\end{align*}

\end{lemma}

\begin{proof}
As in the proof of the preceding Lemma \ref{lemma:02:3}(ii), we use the notation
$ \vec{v} = (i, j, k, l, a, b)$, 
$ \vec{v}^\prime = (i^\prime, j^\prime, k^\prime, l^\prime, a^\prime, b^\prime)$, and we
let
$ Z( \vec{v}) =\{ (a, k), \omega_N ( k, l) \} $
and
$ Z( \vec{v}^\prime) = \{ (a^\prime, k^\prime), \omega_N( k^\prime, l^\prime)\}$.
Consider the sets ( $ s =0, 1, 2$):
\begin{align*}
&V = \{ (i ,j, k, l, a, b) \in [ N]^6 :\  (a, k) \neq \omega_N( k, l) \}\\
& W = \{ (\vec{v}, \vec{v}^\prime)  \in V^2 :\  \vec{v} \neq \vec{v}^\prime \}\\
& W_s = \{ (\vec{v}, \vec{v}^\prime)  \in W :\ 
|Z(\vec{v} ) \cap Z( \vec{v}^\prime) | = s \}
\end{align*}

Next, consider the random variable 
$ F^{a, b}_{i, j, k, l} $
on
$ \mathcal{S}( [ N]^2)$
given by
\begin{align*}
F^{a, b}_{i, j, k, l}(\sigma ) = \left\{
\begin{array}{ll}
1, & \textrm{ if } \sigma (a, k) = \eta_1(b, i)
\textrm{ and } \sigma ( \omega_N (k, l)) = \eta_2(i, j)\\
0, & \textrm{ otherwise. }
\end{array}
 \right.
\end{align*}

Using Lemma \ref{lem:BC} and Markov Inequality, it suffices to show that 
\begin{align*}
\frac{1}{N^4}\mathbb{E} \big( (\sum_V F_{i, j, k, l}^{a, b})^2 \big) = O(N^{0}).
\end{align*}

For 
$ F^{a, b}_{i, j, k, l}(\sigma ) \neq 0 $,
with
$ (i, j, k, l, a, b) \in V $,
there is at most one possible choice for 
$ \sigma( a, k) $ 
and 
$ \sigma( \omega_N (k, l)) $
and 
$ (N^2-2)! $ 
possible choices for the rest of the values of 
$ \sigma $.
 Therefore
 \begin{align*}
\frac{1}{N^4}
\mathbb{E} \big( \sum_V( F_{i, j, k, l}^{a, b})^2 \big)
= 
\frac{1}{N^4}
\mathbb{E} \big(  F_{i, j, k, l}^{a, b} \big)
\leq \frac{1}{N^4} \cdot N^6 \cdot 
\frac{(N^2 -2)!}{N^2!} 
\xrightarrow[N \rightarrow\infty]{}0.
 \end{align*}
 
 If 
 $ (\vec{v}, \vec{v}^\prime) \in W_s $, 
 then for
 $ F^{a, b}_{i, j, k, l} F^{a^\prime, b^\prime}_{i^\prime, j^\prime, k^\prime, l^\prime}
  ( \sigma) \neq 0 $
  there is at most one possible choice for 
$ \sigma \big( Z(\vec{v}) \cup Z(\vec{v}^\prime) \big) $
  and 
$ (N^2 - | Z(\vec{v}) \cup Z(\vec{v}^\prime) |)! $
possible choices for the rest of the values of 
$ \sigma $.
Therefore
\begin{align*}
\frac{1}{N^4}
\mathbb{E} \big( \sum_{W_s} 
F^{a, b}_{i, j, k, l} F^{a^\prime, b^\prime}_{i^\prime, j^\prime, k^\prime, l^\prime} \big)
\leq
\frac{ | W_s^\ast|}{N^4} \cdot
\frac{(N^2 - | Z(\vec{v}) \cup Z(\vec{v}^\prime) |)! }{ N^2!}
\end{align*}
where 
$ W_s^\ast = \{ (\vec{v}, \vec{v}^\prime) \in W_s: 
 F^{a, b}_{i, j, k, l} F^{a^\prime, b^\prime}_{i^\prime, j^\prime, k^\prime, l^\prime}  \neq 0 
 \}$.
Hence it suffices to show that 
\begin{align}\label{ws}
 | W_s^\ast | \leq N^{4 + 2| Z(\vec{v}) \cup Z(\vec{v}^\prime) | }  
= N^{12 -2 | Z(\vec{v}) \cap Z(\vec{v}^\prime) |  }.
\end{align}

For 
$ s = 0 $,
relation (\ref{ws}) is trivially verified, as 
$ W_0^\ast \subset  V \times V \subseteq [N]^{12} $. 
If
$ s \geq 1 $, 
then either
$ (a, k) \in Z(\vec{v}^\prime ) $ 
or 
$ \omega_N (k, l) \in Z( \vec{v}^\prime ) $.
Suppose that 
$ (a, k) \in Z(\vec{v}^\prime )$.
If 
$ (a, k) = (a^\prime, k^\prime)$,
then 
$ F^{a, b}_{i, j, k, l} F^{a^\prime, b^\prime}_{i^\prime, j^\prime, k^\prime, l^\prime}  \neq 0
$
gives that 
$ \eta_1(b, i) = \eta_1(b^\prime, i^\prime)$,
so
$(b, i) = (b^\prime, i^\prime)$.
It follows that 
$(\vec{v}, \vec{v}^\prime)$
is uniquely determined by 
$ (i, j, k, l, a, b, l^\prime, a^\prime)$, 
that is 
$ | W_s^\ast | \leq N^8 $, 
which implies (\ref{ws}).
If 
$ (a, k) = \omega_N ( k^\prime, l^\prime) $
then again 
$ F^{a, b}_{i, j, k, l} F^{a^\prime, b^\prime}_{i^\prime, j^\prime, k^\prime, l^\prime}  \neq 0
$
gives that 
$ \eta_1(b, i) = \eta_2 (i^\prime, j^\prime)$
so 
$ (i^\prime, j^\prime) = \eta_2^{-1} \circ \eta_1 (b, i) $.
It folows that 
$(\vec{v}, \vec{v}^\prime)$
is uniquely determined by 
$ (i, j, k, l, a, b, a^\prime, b^\prime)$, 
that is 
$ | W_s^\ast | \leq N^8 $, 
which implies (\ref{ws}).

The case 
$ \omega_N ( k, l) \in Z ( \vec{v}^\prime) $
is similar.
 \end{proof}

	\begin{proof}[Proof of Theorem \ref{main:2} ]
		
From Theorem \ref{thm:3:1}, we can assume that 	
$ G_N^{ \sigma_N } $
and
$ G_N^{\tau_N} $
are both asymptotically circular with individually zero infinitesimal $*$-distribution. Then, using the free moment-cumulant expansion, if suffices to show that 
$ \mathcal{V}_{\overrightarrow{\sigma_N}} ( \pi) = o(N^{-1}) $
unless 
$ \pi $ 
is non-crossing and
$ \sigma_{k, N} = \top \circ \sigma_{l, N} \circ \top $
whenever
$ \pi(k) = l $.

As in the proof of Theorem \ref{thm:3:1}, eventually by modifying $m $,
we can assume that 
$ \pi $
does not have any blocks of the type
$ (k, k+1)$ 
such that 
$ \sigma_{k, N} = \top \circ \sigma_{l, N} \circ \top $.
Next, using Lemma \ref{lemma:02:1} in the same way Lemma \ref{lemma:01:1} was used in the proof of Theorem \ref{thm:3:1}, we can furthermore assume that 
$ \pi $ 
does not have any blocks with consecutive elements. 
It follow that 
$ \pi $ 
must have at least one crossing, and, as shown in the proof of Theorem \ref{thm:3:1}, we can suppose without loss of generality that there is some 
$ b \in [m-2]$ 
such that
$ \pi(1) = b+1 $ 
and
$ \pi(b) = m $.

Let us denote 
$ S =\{ (\sigma_N)_N, (\top \circ \sigma_N \circ \top)_N \} $,
$ T = \{ (\tau_N)_N, (\top \circ \tau_N \circ \top)_N \} $
and 
$ T_1 = \{ (\id_N)_N \} \cup T $.

Suppose first that 
$ m = 4 $. 
It suffices to show that 
the following result holds true for any sequences
$ (\sigma_{s, N})_N \in S\cup T_1 $, 
$ s =1, \dots, 4 $:
\begin{align}\label{eq:7}
\big|\{ (i, j,k,l) \in [ N]^4 :\ \sigma_{1, N} (i, j) = \top \circ& \sigma_{3, N}  ( k, l)  \textrm{ and }\\
\nonumber 
& \sigma_{2, N} (j, k) = \top \circ \sigma_{4, N}( l, i)\} \big| = o(N^2).
\end{align}

Moreover, at least one of the sequences 
$(\sigma_{s, N})_N $ 
	should be from the set $ S $ and at least one from the set $ T_1$, so it suffices to show (\ref{eq:7}) for the following three cases: 
\begin{enumerate}
	\item[(a)] one of the 
	$(\sigma_{s, N})_N $
	is from the set 
	$ S $ 
	and three are from 
	$ T_1 $
	\item[(b)] two of the 
	$(\sigma_{s, N})_N $
are from the set 
	$ S $ 
	and two are from 
	$ T_1 $
	\item[(c)] three of the 
	$(\sigma_{s, N})_N $
are from the set 
	$ S $ 
	and one is 
	$ (\id_N)_N$.
\end{enumerate}

For case (a), via a circular permutation of the set
$ \{1, 2, 3, 4\}$,
we can suppose that 
$ ( \sigma_{2, N})_N \in S $. 
Then, for 
\begin{align*}
f_N (i, j) &= \left(j, (\pi_1\circ \sigma_{3, N}^{-1} \circ \top \circ \sigma_{1, N}) (i, j) \right)\\
g_N (i, j) &= \top \circ \sigma_{4, N} \left(\pi_2\circ \sigma_{3, N}^{-1} \circ \top \circ \sigma_{1, N}) (i, j), i \right)
\end{align*}
Lemma \ref{lemma:02:2}(i) gives that
(\ref{eq:7}) holds true for any 
$ (\tau_N)_N $ 
and almost surely for 
$( \sigma_N)_N $.

 For case (b), note first that it suffices to show 
 (\ref{eq:7}) when 
$ \{i, k \} \cap \{j, l\} = \emptyset $.
 Indeed, if
  $ i=j $,
   then there are at most 
 $ N $ possible choices for the pair
 $ (i, j) $, 
 each determining at most one possible choice for 
 $ (k, l)  = \sigma_{3, N}^{-1} \circ \top \circ \sigma_{1, N}(i, j) $;
  the other cases are similar.

 Next, without loss of generality, we can further assume that 
  $ ( \sigma_{1, N})_N \in T_1 $
  and 
  $ (\sigma_{2, N})_N \in S $.
 
 Suppose that
 $ (\sigma_{4, N})_N \in S $, 
 Hence
 $ \sigma_{4, N} = \omega_N \circ \sigma_N \circ \omega_N $,
 with
 $ \omega_N $ 
 either identity or matrix transpose. 
 Let
 \begin{align*}
& f_N (i, j) =  (j, k) = \left(j, \pi_1\circ \sigma_{3, N}^{-1} \circ \top \circ \sigma_{1, N} (i, j) \right) \\
 & g_N(i, j)  = \omega_N (l, i) = \omega_N \left ((\pi_2\circ \sigma_{3, N}^{-1} \circ \top \circ \sigma_{1, N}) (i, j), i\right) \\
& \eta_N  = \top \circ \omega_N 
 \end{align*}
 Lemma \ref{lemma:02:3}(i) gives that (\ref{eq:7}) holds true if
 $ f(i, j) \neq g(i, j) $, 
 so it suffices to show (\ref{eq:7}) for 
 $ (i, j, k, l) $
 such that 
 $ (j, k) = \omega_N(l, i )$, 
 which, since
 $ i \neq j $
 implies 
 $ \omega_N = \id_N $, 
 $ i = k $ 
 and 
 $ j = l $.
 It suffices then to show that 
 \begin{align}\label{eq:8}
 \big| \{ (i, j)\in [ N ]^2 : \ \sigma_{1, N} (i, j) = \top \circ \sigma_{3, N} (i, j)\} | = o(N^2) 
 \end{align}
 holds true almost surely for 
 $ (\sigma_{1, N})_N, (\sigma_{3, N})_N \in \{ (\id_N)_N, (\tau_N)_N \}$.

If  one of  
$ (\sigma_{1, N})_N, (\sigma_{3, N})_N $
is 
$ (\tau_N)_N $,
and the other is 
$ (\id_N)_N $,
 then the conclusion follows from Lemma \ref{lemma:02:1}(i). If 
 $ \sigma_{1, N} = \sigma_{3, N} $,
 then equation (\ref{eq:8}) gives that 
 $ (i, j) \in \sigma_{3, N}^{-1} ( \{ (t, t): t \in [ N ]\})$,
 hence there are at most $ N $ possible choices for 
 $(i, j)$,
 and the conclusion follows.
 
 Finally, suppose that 
 $ (\sigma_{3, N})_N \in S $. 
 Then let
 $ \sigma_{3, N} = \omega_N \circ \sigma_N \circ \omega_N $
 with 
 $ (\omega_N)_N $
 either the identity permutation or the matrix transpose.
 For tuples 
 $(i, j, k, l)$ 
 with 
 $ j =l $, 
 according to Lemma \ref{lemma:02:1},
 the condition
 $ \sigma_{1, N} (i, j) = \sigma_{3, N}(k, l) $
 is satisfied by 
 $ o(N^2) $ 
 tuples  for any 
  $ (\sigma_{1, N})_N  $
  almost surely
 $ ( \sigma_N)_N $.

Let 
$ D_1 = \{ ( i, j, k, l) \in [ N ]^4:\ \{ i, k\} \cap \{ j, l\} = \emptyset \textrm{ and } j \neq l 
\}$ 
and define the random variable
\begin{align*}
 I_{i, j, k, l}( \sigma)= \left\{ 
\begin{array}{ll}
1 & \textrm{if } \sigma (i, j) = \top \circ  \omega_N \circ \sigma \circ \omega_N ( k, l)  \textrm{ and }
 \sigma (j, k) = \top \circ \sigma_{4, N}( l, i)\\
 0 & \textrm{otherwise}.
\end{array} \right. 
\end{align*}
 Using Markov Inequality and Lemma \ref{lem:BC}, the result is implied by
 \[ \mathbb{E}
  \big( \sum_{(i, j, k, l) \in D_1}  I_{i, j, k, l} \big) = O(N^0). \]
  
 On the other hand, for 
 $ I_{i, j, k, l}( \sigma) = 1 $, 
 there is one possible choice for 
 $ \sigma( \omega_N (k, l)) $
 and 
 $ \sigma( j, k) $
 and 
 $ (N^2 -2 )! $
 possible choices for the rest of the values of
 $ \sigma$. 
 Therefore
 \begin{align*}
 \mathbb{E}
 \big( \sum_{(i, j, k, l) \in D_1}  I_{i, j, k, l} \big)
 < N^{4} \cdot \frac{(N^2 -2)!}{N^2!} 
 \xrightarrow[N \rightarrow \infty ]{}1.
 \end{align*}
 
 For case (c), we can suppose, without loss of generality, that 
 $ \sigma_{1, N} = \id_N $ 
 and 
 $ \sigma_{3, N} = \sigma_N $.
 I.e. we have to show that 
 \begin{align}\label{eq:81}
 \big| \{ (i, j, k, l)  :\ 
  \sigma_N (k, l) = (j, i) 
  \textrm{ and } 
  \sigma_{2, N} (j, k) =  \top \circ \sigma_{4, N} ( l, i)
  \} \big| = o(N^2)
 \end{align}
 holds true almost surely for 
 $ ( \sigma_N)_N $ 
 where ($ s =1, 2 $)
 $ \sigma_{s, N} = \omega_{s, N} \circ \sigma_N \circ \omega_{s, N}$
 and
 $ ( \omega_{s, N})_N $ 
 is either identity transforms or matrix transposes.

 As discussed above, (\ref{eq:81}) is trivially verified if 
 $ i = j $, $ j = k $ or $ i = l $. 
 Furthermore, if 
 $ k = i $, 
 or
 $ l = j $,
 then (\ref{eq:81}) follows applying Lemma \ref{lemma:02:1}(i) 
 for the condition 
 $ \sigma_N (k, l) = (j, i)$. 
 Hence, denoting 
 $ D_2 = \{ (i, j, k, l) \in [ N ]^4:\  i, j, k, l \textrm{ distinct}  \} $
 it suffices to show that (\ref{eq:81}) holds true for 
 $ (i,j, k,l) \in D_2 $,

 Define the random variable on 
 $ \mathcal{S}( [ N]^2)$:
 \begin{align*}
 J_{i,j, k, l}( \sigma) = 
 \left\{ 
 \begin{array}{ll}
 1 & \textrm{ if } \sigma(k, l) = (j, i)
  \textrm{ and }  \sigma_{1, N} ( j, k) = t \circ \sigma_{2, N} (l, i) \\
 0 & \textrm{ otherwise }
 \end{array}\right.
 \end{align*}
 where
 $ \sigma_{s, N} = \omega_{s, N} \circ \sigma \circ \omega_{s, N} $
 for 
 $ s = 1, 2 $.
 
 As before it suffices to show that 
 $ \displaystyle \mathbb{E} \big(\sum_{ (i, j, k, l) \in D_2} J_{i, j, k, l} \big) = O(N^0)$.
 
 Note that if 
 $ (i, j, k, l) \in D_2 $, 
 then
 $ (k, l) $, $ \omega_{1, N} ( j, k) $ 
 and 
 $ \omega_{2, N} (l, i) $
 are distinct.
 Hence, for
 $ J_{i, j, k, l} \neq 1 $, 
 there is one possible choice for 
 $ \sigma(k, l) $,
 $N^2 -1 $
 possible choices for 
 $ \sigma( \omega_{1, N}(j, k))$,
 one possible choice for
 $ \sigma( \omega_{2, N}(l, i))$
 and
 $ (N^2 - 3)! $
 possible choices for the rest of values of 
 $ \sigma $.
 Therefore
 \begin{align*}
 \mathbb{E} 
 \big(\sum_{ (i, j, k, l) \in D_2} J_{i, j, k, l} \big) 
 < 
 N^4 \cdot
 \frac{(N^2-1)\cdot (N^2 -3)!}{N^2!}
 \xrightarrow[N\rightarrow \infty]{}1.
 \end{align*}
 
 Next, suppose that 
 $ m > 4 $.
 As in the proof of Theorem \ref{thm:3:1}, it suffices to show that 
 $ \mathfrak{a}( \{ 1, b, b+1, m\}) < 0 $, 
 i.e. that the following result holds true for any
 $ (\sigma_{s, N})_N \in S\cup T_1 $, 
 where
 $ s \in \{1, b, b+1, m \} $:
 \begin{align}\label{eq:10}
 \big|\{ (i, j,k,l, a, b) \in [ N]^6:\ 
 \sigma_{1, N}(i, j) = & \top \circ  \sigma_{b+1, N} (k, l) 
 \textrm{ and }\\
 \nonumber   & \sigma_{b, N} (a, k) = \top \circ \sigma_{m, N}(b, i) \}\big| = o(N^3).
 \end{align}
 
 We shall prove the statement above by analysing the same cases (a), (b) and (c) as in the setting 
 $ m = 4$.
 
 For case (a), via a circular permutation of the set 
 $[ m ] $ and taking adjoints, we can suppose that  
 $ (\sigma_{b, N} )_N  = ( \sigma_N)_N $.
Putting 
\begin{align*}
 \phi (i, j, a)& = (a, k)= (a, \pi_1 \circ \sigma_{b+1, N}^{-1} \cdot \top \cdot \sigma_{1, N} (i, j))\\
 \psi(i, j, b) &= \sigma_{m, N} (b, i)
 \end{align*}
 note that 
 $ \phi(i, j,a) = \phi(i^\prime, j^\prime, a^\prime) $
 and 
 $ \psi(i, j, b)= \psi(i^\prime, j^\prime, b^\prime) $
 implies 
 $ (i, a, b) = (i^\prime, a^\prime, b^\prime)$,
 and the conclusion follows from Lemma \ref{lemma:02:2}(ii).
 
 For case (b) we can assume, without loss of generality,  that 
 $ ( \sigma_{1, N})_N \in T_1 $
 and 
 $ (\sigma_{b, N})_N \in S $. 
 It suffices to distinguish two subcases, when
 $ (\sigma_{b+1, N})_N \in S $,
 respectively when
 $ (\sigma_{m, N})_N \in S $.
 
 Suppose that 
 $ (\sigma_{m, N})_N \in S $, 
 that is 
 $ \sigma_{m, N} = 
 \omega_N \circ \sigma_{N} \circ \omega_N $
 with 
 $ \omega_N $
 either the identity or the matrix transpose.
 Applying Lemma \ref{lemma:02:3}(ii) for
 $ h(i, j) = k = \pi_1 \circ \sigma_{b+1, N}^{-1} \circ \sigma_{1, N} (i, j)$, 
 we obtain that relation (\ref{eq:10}) with the extra condition 
 $ (a, k) \neq \omega_N (b, i)$ 
 is satisfied for all
 $ (\sigma_{1, N})_N , (\sigma_{b+1, N})_N \in T_1 $
 and almost surely for 
 $ (\sigma_N)_N \in S $. 

Assume that 
$ (a, k) = \omega_N (b, i)$.
If 
$ \omega_N  = \id_N $, 
then the equality 
 $  \sigma_{b, N} (a, k) = t \circ \sigma_{m, N}(b, i) $
 gives that 
 $ \sigma_N (a, k) = t \circ \sigma_N (a, k)$,
 hence 
 $ (a, k) \in \sigma_{N}^{-1}  \{ (s, s):  s \in [  N ]\} $
 so there are at most 
 $ N^2$ possible choices for 
$(a, k,l)$, 
which uniquely determines 
$(i,j, k, l, a, b)$.
If 
$ \omega_N $ 
is the matrix transpose, then the equality 
$  \sigma_{b, N} (a, k) = \top \circ \sigma_{m, N}(b, i) $
gives that 
$ \sigma_N (a, k) = \sigma_N (k, a)$, 
that is 
$ a = k $, 
so again there are at most 
$ N^2 $ 
possible choices for the triple
$ (a, k,l)$.

Suppose that 
$ ( \sigma_{b+1, N})_N \in S $,
that is 
$ \sigma_{b+1, N} = 
\omega_N \circ \sigma_{N} \circ \omega_N $
with 
$ \omega_N $
either the identity or the matrix transpose. 
Applying Lemma \ref{lemma:02:4}, 
we obtain that relation (\ref{eq:10}) with the extra condition 
$ (a, k) \neq \omega_N (k,l)$ 
is satisfied for all
$ (\sigma_{1, N})_N , (\sigma_{b+1, N})_N \in T_1 $
and almost surely for 
$ (\sigma_N)_N \in S $.
 
Assume that
$ (a, k) = \omega_N (k, l)$. 
Then 
$ (a, k, l)$ 
is uniquely determined by 
$(a, k)$.
But
$ (i, j, k, l, a, b) $
is uniquely detemined by 
$ (a, k,l)$, so 
(\ref{eq:10}) follows.

 For case (c), we can assume that 
 $ \sigma_{1, N} = \id_N $
 and 
 $ \sigma_{b+1, N} = \sigma_N $.
 I.e. we have to show that the following relation holds true almost surely for 
 $ ( \sigma_N)_N $:
 \begin{align}\label{eq:c}
 \big| \{ (i, j, k, l, a, b)\in [ N]^6:\  \sigma_N (k, l) = (j, i) \textrm{ and } \sigma_{b, N}(a, k) = \top \circ \sigma_{m, N} (& b ,  i) \}
 \big| \\ \nonumber   & = o(N^3)
 \end{align}
 where
 $ \sigma_{b, N} = \omega_{1, N} \circ \sigma_N \circ \omega_{1, N} $
 and 
 $ \sigma_{m, N} = \omega_{2, N} \circ \sigma_N \circ \omega_{2, N} $
 with 
 $ ( \omega_{s, N})_N $
  either the identity permutation or the matrix transpose. 
  
Note that given 
$ ( \sigma_N)_N $,
the tuple 
$( i, j,k, l, a, b) $ 
is uniquely determined by either of the triples
$ (k, l, a) $ 
and
$ (k, l, b) $
hence
property (\ref{eq:c}) is verified under one of the extra conditions 
$ (k, l)\in\{ \omega_{1, N}(a, k), \omega_{2, N}(b, i) \} $.

If 
$ \omega_{1, N}(a, k) = \omega_{2, N}(b, i)  = (u, v)$, 
then
$ \sigma_{2, N}(a, k) = \top \circ \sigma_{4, N}(b, i)$ 
gives that 
$ \omega_{1, N} (\sigma_N (u, v) ) = \top \circ \omega_{2, N}
(\sigma_N (u, v)) $, 
so either 
$ \omega_{1, N} = \top \circ \omega_{2, N} $
or 
$ \top \circ \sigma_N (u, v) = \sigma_N (u, v) $.
But 
 $ \omega_{1, N} = \top \circ \omega_{2, N} $, 
gives 
$ (k, a) = (b, i) $, 
so 
$ ( i,j, k, l, a, b) $
is uniquely determined by 
$ (k, l)$.
Also, if 
 $ \top \circ \sigma_N (u, v) = \sigma_N (u, v) $, 
 then there are at most 
 $ N $ 
 possible choices for 
 $ \sigma (u, v) $, 
 that is for 
 $ (a,k)$
 so there are at most 
 $ N^2 $
 possible choices for 
 $ (k, l, a)$. 
 
 Denote
 $ V = \{ (i, j, k,l, a, b) \in [ N ]^6 :\ (k, l), \omega_{1, N}(a, k), \omega_{2, N}(b, i) \textrm{ distinct}\} $.
 
 Consider the random variable
 \begin{align*}
 J^{a, b}_{i, j, k, l} ( \sigma) = \left\{
 \begin{array}{ll}
 1 & \textrm{ if }  \sigma (k, l) = (j, i) \textrm{ and } \sigma_1(a, k) = \top \circ \sigma_2 (b, i)\\
 0 & \textrm{ otherwise }
 \end{array}
 \right.
 \end{align*}
where 
$ \sigma_1 = \omega_1 \circ \sigma\circ \omega_1 $
and
$ \sigma_2 = \omega_2 \circ \sigma \circ \omega_2 $,
and applying Markov Inequality and Lemma \ref{lem:BC}, it suffices to show that 
\begin{align}\label{eq:c:e}
N^{-4} \mathbb{E} \big( ( \sum_V J_{i, j, k, l}^{a, b} )^2\big) = O(N^0).
\end{align}

For 
$ J^{a, b}_{i, j, k, l}( \sigma) \neq 0 $
there is one possible choice for 
$ \sigma(k, l) $, 
less than 
$ N^2 $ 
possible choices for 
$ \sigma(\omega_{1, N} (a, k)) $,
 one choice for 
$ \sigma(\omega_{2, N} (b, i))$
and
$ (N^2 -3 )! $ 
possible choices for the other values of 
$ \sigma $. 
Therefore
\begin{align*}
N^{-4} \cdot \mathbb{E} \big ( \sum_V (J_{i, j,k, l}^{a, b})^2 \big)
= 
N^{-4} \cdot \mathbb{E} \big ( \sum_V J_{i, j,k, l}^{a, b} \big)
< N^{-4} \cdot N^6 \frac{N^2 \cdot (N^2 -3)! }{N^2!} 
\xrightarrow[ N \rightarrow \infty ] {} 0.
\end{align*}
  To simplify the writing, let us denote by 
  $ \vec{v}  = (i, j,k, l, a, b)$,
  $ \vec{v}^\prime = (i^\prime, j^\prime, k^\prime, l^\prime)$,
  $ Z(\vec{v}) = \{ (k, l), \omega_{1, N}(a, k), 
  \omega_{2, N}(b, i)\}$
  and
  $ Z(\vec{v}^\prime) = \{ (k^\prime, l^\prime), 
  \omega_{1, N}(a^\prime, k^\prime), \omega_{1, N}(b^\prime, i^\prime)\}$.
  Note that if
  $ Z(\vec{v})= Z(\vec{v}^\prime)$, 
  then
  $ (k, l, a) = (k^\prime, l^\prime, a^\prime) $,
  so 
  $ \vec{v} = \vec{v}^\prime$.
  
  Denoting
  \begin{align*}
  & W =\big\{ (\vec{v}, \vec{v}^\prime ) \in V^2:\ \vec{v} \neq \vec{v}^\prime \big\} \\
  & W_1 = \big\{ (\vec{v}, \vec{v}^\prime ) \in W  :\
  Z(\vec{v}) \cap Z(\vec{v}^\prime)  = \emptyset \big\}\\
  & W_2 =\big\{ (\vec{v}, \vec{v}^\prime ) \in W :\ 
  | Z(\vec{v}) \cap Z(\vec{v}^\prime) | = 1 \big\} \\
  & W_3 = \big\{ (\vec{v}, \vec{v}^\prime ) \in W :\ 
  | Z(\vec{v}) \cap Z(\vec{v}^\prime) | = 2 \big\}\\
  \end{align*}
it follows that 
\begin{align*}
\mathbb{E} \big( \sum_W J_{i, j, k, l}^{a, b} \cdot
J_{i^\prime, j^\prime, k^\prime, l^\prime}^{a^\prime, b^\prime} ) 
= \sum_{ s =1}^3 \mathbb{E} \big(
\sum_{W_s} J_{i, j, k, l}^{a, b} \cdot
J_{i^\prime, j^\prime, k^\prime, l^\prime}^{a^\prime, b^\prime}
\big)
\end{align*} 

  Let us remind that, as discussed above, 
  for
  $ J^{a, b}_{i, j, k, l} (\sigma)\neq 0 $,
  there are at most 
  $ N^2 $ 
  possible choices for 
  $ \sigma ( Z(\vec{v}))$.
 So, for 
 $ J_{i, j, k, l}^{a, b} \cdot
 J_{i^\prime, j^\prime, k^\prime, l^\prime}^{a^\prime, b^\prime} ( \sigma ) \neq 0  $
 there are at most 
 $ N^4$ 
 possible choices for 
 $ \sigma ( Z(\vec{v}) \cup Z ( \vec{v}^\prime) ) $
 and at most 
$ (N^2 - | Z(\vec{v}) \cup Z( \vec{v}^\prime)|)! $
possible choices for the rest of the values of 
$ \sigma $,
therefore
\begin{align*}
 \mathbb{E} \big(
\sum_{W_s} J_{i, j, k, l}^{a, b} \cdot
J_{i^\prime, j^\prime, k^\prime, l^\prime}^{a^\prime, b^\prime}
\big) \leq | W_s | \cdot N^{-4} \cdot \frac{N^4 \cdot (N^2 - | Z(\vec{v}) \cup Z( \vec{v}^\prime)|)! } {N^2 !} 
\end{align*}
but 
$| Z(\vec{v}) \cup Z(\vec{v}^\prime )  | 
= 6 - | Z(\vec{v}) \cap Z(\vec{v}^\prime )  | $,
 so it suffices to show that 
\begin{align}\label{eq:c:sets}
\log_N | W_s | < 
 12 - 2| Z(\vec{v}) \cap Z(\vec{v}^\prime )  | .
\end{align}  
 
 For 
 $ s = 1 $, 
 we have that 
 $ W_1 \in V \times V \in [N]^{12} $, 
 so (\ref{eq:c:sets}) holds true. 
 For 
 $ s = 2 $, 
 since 
 $ | Z(\vec{v}) \cap Z(\vec{v}^\prime )  | = 1 $,
 at least two components of 
 $ \vec{v} $ are equal to two components of 
 $ \vec{v}^\prime $, 
 hence 
 $ |W_2 | \leq N^{10}$,
 which implies (\ref{eq:c:sets}).
 
 If
 $ s = 3 $, 
 note that all subsets with two elements of 
 $ Z(\vec{v})$ 
 contain four components of 
 $ \vec{v} $,
  thus (\ref{eq:c:sets}) being satisfied,
 except for
 $ \{ (k, l), \omega_{1, N} (a, k) \}$
 If 
 $ \{ (k, l), \omega_{1, N} (a, k) \}
 = \{ (k^\prime, l^\prime), \omega_{1, N} (a^\prime, k^\prime) \} $, 
 then, for
$ J_{i, j, k, l}^{a, b} \cdot
J_{i^\prime, j^\prime, k^\prime, l^\prime}^{a^\prime, b^\prime} (\sigma) \neq 0 $,
there are at most 
$ N^2 $ 
possible choices for 
$ \sigma ( Z(\vec{v})) $,
each giving one possible choice for
$ \sigma ( \{ (k, l), \omega_{1, N}( a, k) \} ) $,
Hence for 
$ \sigma( Z( \vec{v}^\prime)) $
and 
$( N^2 - 4)! $ 
possible choices for the rest of the values of 
$ \sigma$.
Therefore, denoting
$ W_4 =  \big\{ (\vec{v}, \vec{v}^\prime) \in W :\
\{ (k, l) , (a, k) \} = \{ (k^\prime, l^\prime), (a^\prime, k^\prime)\}   \big \}$,
 we have that 
\begin{align*}
\mathbb{E} \big(\sum_{W_4} J_{i, j, k, l}^{a, b} \cdot
J_{i^\prime, j^\prime, k^\prime, l^\prime}^{a^\prime, b^\prime}   \big) 
< N^{9} \cdot N^{-4} \cdot \frac{N^2 \cdot (N^2 - 4)!}{N^2!}
\xrightarrow[ N \rightarrow\infty]{}0,
\end{align*}
and the conclusion follows.
	\end{proof}


\section{Joint Infinitesimal Distribution of a Gaussian Random Matrix and its Transpose}

In this section, we investigate the their joint infinitesimal distribution of Gaussian random matrix and its transpose, we describe their joint infinitesimal free cumulants in the following theorem. In particular, we show that Gaussian random matrix and its transpose are not asymptotically infinitesimally free. 

\begin{thm}\label{main 3}
For each positive integer $ N $, 
 consider 
 $ G_N $ a complex Gaussian random matrix and 
 $ G_N^{\top} $ 
 its matrix transpose.
 The asymptotic values (as $ N \rightarrow \infty $) of the infinitesimally free joint cumulants of 
 $ G_N $ 
 and 
 $ G_N^{\top} $ 
 are computed according to the following rule (here each 
 $ \varepsilon_j $  is either the identity or the matrix transpose):
 \begin{align*}
 \lim_{N \rightarrow\infty}\kappa^\prime_p ( G_N^{\varepsilon_1}, G_N^{\varepsilon_2}, \dots, G_N^{\varepsilon_p}) = \left\{
 \begin{array}{ll}
 1 & \textrm{ if } p = 2m, \varepsilon_s \neq \varepsilon_{m+s}, \textrm{ for } s \in [m]\\
 0 & \textrm{ otherwise. }
 \end{array}  \right.
 \end{align*}
\end{thm}

Before proceeding with the proof of the Theorem above, notice that simple computations give the following particular cases of Lemmata \ref{lemma:02:2}, \ref{lemma:02:3}, \ref{lemma:02:4}.

\begin{remark}\label{remark:c1-2}
 Suppose that for 
 $ 1 \leq s \leq 4 $,  
 $ \sigma_{s, N} $ is either the identity or the matrix transpose in $\mathcal{S}([ N]^2 )$
 and denote
 \begin{align*}
 C_1( \sigma_1, \dots, \sigma_4) 
 = \big| \{ (i, j, k, l) \in [ N ]^4:\, & \sigma_1( i, j) = \top \circ \sigma_3 ( k, l), \\  &\sigma_2 ( j, k) = \top \circ \sigma_4 (l, i) \} \} \big| \\
 C_2( \sigma_1, \dots, \sigma_4) = 
 \big| \{  (i,j, k, l, a, b) \in [ N ]^6:\, & \sigma_1( i, j) = \top \circ \sigma_3 (k, l), \\ &  \sigma_2( a, k) = \top \circ \sigma_4 (b, i) 
 \}\big|.
 \end{align*}
 Then
 \begin{align*}
 C_1( \sigma_1, \dots, \sigma_4 ) = \left\{ 
 \begin{array}{ll}
 N^2 & \textrm{ if } \sigma_1 \neq \sigma_3 \textrm{ and } \sigma_2 \neq \sigma_4\\
 N & \textrm{ otherwise,}
 \end{array}
 \right.
 \end{align*}
and
 \begin{align*}
C_2( \sigma_1, \dots, \sigma_4 ) = \left\{ 
\begin{array}{ll}
N^3 & \textrm{ if } \sigma_1 \neq \sigma_3 \textrm{ and } \sigma_2 \neq \sigma_4\\
N^2 & \textrm{ otherwise.}
\end{array}
\right.
\end{align*}
\end{remark}
\begin{proof}
Since the identity commutes with the transpose, for 
$ 1 \leq s, t \leq 4 $,
we have that
$ \sigma_s^{-1} \circ \top \circ \sigma_t (i, j) = \left\{
\begin{array}{ll}
(i, j) & \textrm{ if } \sigma_s \neq \sigma_t \\
(j, i) & \textrm{ if } \sigma_s = \sigma_t.
\end{array} \right.$. 
So, if
$ \sigma_1 \neq \sigma_3 $ 
and
$ \sigma_2 \neq \sigma_2 $, 
then
\begin{align*}
C_1(\sigma_1, \dots, \sigma_4 ) &= \big| \{ (i, j, k, l) \in [ N]^4:\ (i, j)= (k, l) ,\  (j, k)= (l, i)
\} \big| \\
& = \big| \{ (i, j, k, l)\in[ N]^4 :\ i=k , \  j = l 
\} \big|= N^2
\end{align*}
and
\begin{align*}
C_2(\sigma_1, \dots, \sigma_4 ) & = 
\big|\{ 
(i, j, k, l, a, b) \in [ N ]^6:\ (i, j)= ( k, l) ,\ (a, k) = (b, i)
\}\big|\\
& = \big|\{ (i, j, k, l, a, b) \in [ N ]^6:\ i = k, \ j= l, \ a = b 
\} \big| = N^3.
\end{align*}

If 
$ \sigma_1 = \sigma_3 $, 
then 
$  i = l$ 
and 
$ j = k $, 
hence
$ \sigma_2( j, k) = \sigma_4 ( l, i)$ 
gives that 
$ i =j = k =l $, that is 
$ C_1( \sigma_1, \dots, \sigma_4 ) = N $.
Also, 
$ \sigma_2(a, k) = \sigma_4 ( b, i) $
gives that $ (b, i) = \sigma_4 ^{-1} \circ \sigma_2 ( a, j)$,
so $ (i,j,k, l, a, b) $ 
is uniquely determined by 
$(a, j)$, 
that is
$ C_2( \sigma_1, \dots, \sigma_4) = N^2 $. 
The argument for the case 
$ \sigma_2 = \sigma_4 $
 is similar.
\end{proof}

\begin{lemma}\label{lemma:v:pi}
 Let 
 $ n, N $ 
 be a positive integers and suppose that for each 
 $ s \in [ n ] $, 
 $ \sigma_s $ 
 is either the identity or the matrix transpose in 
 $\mathcal{S}([ N]^2 ) $.
Denote 
$ \overrightarrow{\sigma} =(\sigma_1, \dots, \sigma_n)$, 
and, with the notations from Section \ref{section:1}, write
\begin{align*}
	\mathbb{E} \circ \tr \big(G_N^{\sigma_1} \cdot G_N^{\sigma_2} \cdots G_N^{\sigma_n}\big) = \sum_{ \pi \in P_2(n)} \mathcal{V}_{\overrightarrow{\sigma}}(\pi).
\end{align*}
Also, write the set 
$ P_2(m) $ 
as the disjoint union 
\begin{align*}
 P_2(n) = P_2( \overrightarrow{\sigma}, 0) \cup 
P_2( \overrightarrow{\sigma}, 1) \cup
P_2( \overrightarrow{\sigma}, 2)
\end{align*}
where
\begin{align*}
P_2( \overrightarrow{\sigma}, 0) 
& = \{ \pi \in P_2(n):\
\pi \textrm{ is non-crossing and } \sigma_s = \sigma_{ \pi(s)} \textrm { for all } s \in [ n ] \}\\
P_2( \overrightarrow{\sigma}, 1) 
& = \{ \pi \in P_2(n):\ \textrm{ there exists some }
B = \{ i(1), i(2), \dots, i(2m) \}\subseteq  [ n ] 
\textrm{ with }\\
& 
 i(1) < \dots < i(2m)
  \textrm{ such that } 
  \pi(i(s)) = i(m +s)  
  \textrm{ and }
 \sigma_{i(s)} \neq \sigma_{i( m + s )}\\
 &\textrm{ and } \sigma_k = \sigma_{\pi(k)} \textrm{ whenever } k \in [ n ] \setminus B \textrm{ and  the permutation } \sigma \vee 1_B 
 \\
 & \textrm{ ( obtained from $\sigma $ by considering $ B $ as a block) is non-crossing} \big\}\\
P_2( \overrightarrow{\sigma}, 2) 
& = 
P_2 ( n) \setminus ( P_2( \overrightarrow{\sigma}, 0)
\cup P_2( \overrightarrow{\sigma}, 1)  ).
\end{align*}
With the notations above, we have that:
\begin{align}\label{v:pi}
\mathcal{V}_{\overrightarrow{\sigma}}(\pi) = \left\{
\begin{array}{ll}
1 & \textrm { if }
 \pi \in P_2( \overrightarrow{\sigma}, 0) \\
N^{-1} & \textrm{ if } 
\pi \in P_2( \overrightarrow{\sigma}, 1)\\
O(N^{-2}) & \textrm{ if } 
\pi \in P_2( \overrightarrow{\sigma}, 2).
\end{array}
 \right.
\end{align}
\end{lemma}

\begin{proof}
 As in (\ref{v1-1}) from the proof of Theorem \ref{thm:3:1}, if 
 $ \pi(k) = k+1 $ 
 and 
 $ \sigma_k = \sigma_{k+1} $,
 then	
 $ 
 \mathcal{V}_{ \overrightarrow{\sigma}}(\pi) = \mathcal{V}_{\overrightarrow{\sigma}^\prime} ( \pi^\prime)
 $
where
$ \pi^\prime$
and
$ \overrightarrow{\sigma}^\prime $
are obtained by removing
$ (k, k+1) $
respectively
$ \sigma_{k}, \sigma_{k + 1 } $
from
$ \pi$, 
respectively
$ \overrightarrow{\sigma} $. 	
If
$ \pi \in P_2( \overrightarrow{\sigma}, 0) $ 
then  iterating (\ref{v1-1}) 
$ n/ 2 $ times
gives the first part of (\ref{v:pi}).
Moreover, if 
$ \pi \notin P_2( \overrightarrow{\sigma}, 0) $, 
then, (\ref{v1-1}) allows us to assume, without loss of generality that 
$ \sigma_k \neq \sigma_{k+1} $
whenever
$ \pi(k) = k+1 $.

Under all the assumptions above, suppose first that 
$ \pi $ 
is non-crossing. We shall prove (\ref{v:pi}) by induction on 
$ n $. 
If 
$ n =2 $, 
then 
$\pi \in P_2( \overrightarrow{\sigma}, 1) $ 
and
\begin{align*}
\mathcal{V}_{\overrightarrow{\sigma}}(\pi)
= N^{-2} \cdot \big| \{ (i, j)\in [ N]^2:\ 
\sigma_1(i, j) = \top\circ  \sigma_1 \circ \top ( j, i) 
\}\big| = N^{-1},
\end{align*}
hence (\ref{v:pi}) holds true.

If 
$ n \geq 4 $, 
then 
$\pi \in P_2( \overrightarrow{\sigma}, 2) $. 
On the other hand, since
 $ \pi $ 
 is non-crossing, it
 has at least one block which is a segment. Via a circular permutation of the set 
 $ [ n] $,
 we can suppose without loss of generality that
 $ \pi(n-1) = n $, 
 which furthermore implies that
 $ \sigma_{n-1} = \top \circ \sigma_N $.
 Then the condition
 \begin{align*}
 \sigma_{n-1}(i_{n-1}, j_{n-1}) = \sigma_n \circ \top ( i_n, j_n)
 \end{align*}
 gives that 
 $ i_{n-2} = i_{n} $ .
 Therefore, denoting by 
  $\pi^\prime $, 
  respectively by
  $ \overrightarrow{\sigma}^\prime $
  the restrictions of
   $ \pi $,
 respectively of 
 $ \overrightarrow{\sigma} $
 to the set 
 $ [n-2]$,
 we then have that 
 \begin{align*}
 \mathfrak{a}_{ \pi^\prime, \overrightarrow{\sigma}^\prime}\leq
 \mathfrak{a}_{ \pi, \overrightarrow{\sigma}}([n-2])
 \end{align*}
and
\begin{align*}
 \mathfrak{a}_{ \pi, \overrightarrow{\sigma}} \leq
  \mathfrak{a}_{ \pi, \overrightarrow{\sigma}}([ n-2]) -1
\end{align*} 
hence (\ref{v:pi}) follows.

Next, suppose that 
$ \pi $
is crossing. As in the proof of Theorem \ref{thm:3:1}, we can assume that 
$ \pi(1) = b+1 $ 
and 
$ \pi(b) = d $
for 
$ 1 < b<b + 1 < d \leq n $.
If 
$ d \neq n $, 
then
 $\pi \in P_2( \overrightarrow{\sigma}, 2)$ 
 and, as in the proof on Theorem \ref{thm:3:1}, in this case we have that  
$ \mathfrak{a}_{\pi, \overrightarrow{\sigma}} \leq -2 $,
so (\ref{v:pi}) holds true.

If 
$ d =n $ 
and 
$ \sigma_1 = \sigma_{b+1} $ 
or 
$ \sigma_{b} = \sigma_{d} $,
then again
$\pi \in P_2( \overrightarrow{\sigma}, 2) $; 
 Remark \ref{remark:c1-2}  and Lemma \ref{lemma:01} give that 
$ \mathfrak{a}_{\pi, \overrightarrow{\sigma}} \leq -2 $,
so (\ref{v:pi}) holds true.

Suppose that 
$ d = n $, 
$ \sigma_1 \neq \sigma_{b+1}$, 
$ \sigma_b \neq \sigma_d $ 
and if
$ \pi(k) = k+1 $
we have that 
$ \sigma_k \neq \sigma_{k+1} $.
If 
$ n =4 $, 
then 
$ \pi \in P_2( \overrightarrow{\sigma}, 1) $, 
 and Remark \ref{remark:c1-2} gives that 
$ \mathfrak{a}_{\pi, \overrightarrow{\sigma}} = - 1$,
so (\ref{v:pi}) holds true in this case. 
Suppose then that 
$ n > 4 $. 
If there is some
 $ s\in  \{ 2, \dots, b-1\} $ 
 or
 $ s \in \{ b+2, \dots, n-1 \} $
 such that 
 $ \pi(s) = s+1  $,
 then 
 $ \pi \in P_2( \overrightarrow{\sigma}, 2) $; on the other hand, we assumed that 
 $ \sigma_s \neq \sigma_{s+1} $
 so 
 $ \mathfrak{a}_{\pi, \overrightarrow{\sigma}} (\{ s, s+1\}) = -1 $
 and Lemma \ref{lemma:01} gives that 
 $\mathfrak{a}_{\pi, \overrightarrow{\sigma}} \leq -2 $
 so (\ref{v:pi}) holds true. The same argument remains valid if there exists 
 $ a^\prime < b^\prime < c^\prime <d^\prime $
 elements either of 
 $ \{ 2, \dots, b-1\} $ 
 or of
 $\{ b+2, \dots, n-1 \}$
 such that
 $ \pi(a^\prime) = c^\prime $
 and
 $\pi(b^\prime) = d^\prime$
 (see Figure 4). So we can further assume that 
 $ \pi([ b]) =  [ n ] \setminus [ b] $, 
 in particular 
 $ n = 2b $.
 
 Note that the condition 
 $ \pi([ b]) =  [ n ] \setminus [ b ] $
 gives that 
 $ A_{\pi, \overrightarrow{\sigma}} = A_{\pi, \overrightarrow{\sigma}} ([ b] ) $.
 Furthermore, since
 $ \pi(1)= b+1 $ 
 and 
 $ \sigma_1 \neq \sigma_{b+1} $
 we have that 
 $ i_{b+1} = i_1 $,
 so
 $  A_{\pi, \overrightarrow{\sigma}} ([ b] )
 =  A_{\pi, \overrightarrow{\sigma}} ([ b -1] ) $.  
 
 If 
 $ \pi \in P_2( \overrightarrow{\sigma}, 1) $,
 then
 \begin{align*}
 A_{\pi, \overrightarrow{\sigma}} &=  \{(i_1, i_2, \dots, i_{2b}) \in [ N]^{2b}: (i_s, i_s+1) = (i_{m+s}, i_{m+s}+1), 
 i_{2b} = i_1 
 \} \\
 & = \{ (i_1, i_2, \dots, i_b) \in [ N]^b 
 \} ,
 \end{align*}
 and since
  $
 \mathfrak{a}_{\pi, \overrightarrow{\sigma}} = 
 \log_N\big|  A_{\pi, \overrightarrow{\sigma}}\big| - b -1 
 $
 it follows that (\ref{v:pi}) holds true.
 

\setlength{\unitlength}{.13cm}
\begin{equation*}
\begin{picture}(22,16)
\put(-20.5,0){1}

\put(-20,3){\circle*{2}} 
\put(-16.5,0){2}

\put(-16,3){\circle*{2}} 

\put(-12.5,0){3}

\put(-12,3){\circle*{2}} 

\multiput(-9 ,3)(2, 0){4}{\circle*{.5}}

\put(-12.5,0){3}

\put(-12,3){\circle*{2}} 


\put(-2.5,0 ){$b-1$}

\put( 0,3){\circle*{2}}

\put(4.5,0 ){$b$}

\put( 5,3){\circle*{2}}

\put(8.5,0 ){$b+1$}

\put( 12,3){\circle*{2}}



\put(16,3){\circle*{2}}

\put(20,3){\circle*{2}}

\multiput(23 ,3)(2, 0){4}{\circle*{.5}}

\put(28,0){$ n-1$}

\put(32,3){\circle*{2}}


\put(36,0){$ n$}

\put(36,3){\circle*{2}}

\put(-16,3){\line(0,1){6}}

\put(16,3){\line(0,1){6}}

\put(-16,9){\line(1,0){32}}


\put(-12,3){\line(0,1){8}}

\put(20,3){\line(0,1){8}}

\put(-12,11){\line(1,0){32}}


\put(0,3){\line(0,1){10}}

\put(32,3){\line(0,1){10}}

\put(0,13){\line(1,0){32}}

\linethickness{.55mm}


\put(-20,4){\line(0,1){3}}

\put( 12,3){\line(0,1){4}}

\put(-20,7){\line(1,0){32}}


\put( 5,3){\line(0,1){12}}

\put(36,3){\line(0,1){12}}

\put( 5,15){\line(1,0){31}}

\end{picture}
\end{equation*}
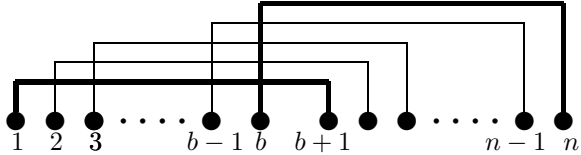
\textbf{Figure 8}. 


 If 
 $ \pi \in P_2( \overrightarrow{\sigma}, 2) $,
 then the set
 $
 \{s\in [ b]:\ \pi(s) \neq b +s \textrm{ or } \sigma_s \neq \sigma_{ b+s } \} 
 $
 is non-void.
Denote by
 $ t $ 
 its smallest element. 
 Then
 $ \pi(t-1) = m+ t -1 $
 and
 $ \sigma_{t-1} \neq \sigma_{m + t -1 } $
 so
 $ i_t = i_{m+t} $. 
 Let 
 $ v = \pi( m + t) $.
 Then
 $ v \in [ b-1] \setminus \{ t -1\}$.
 and
 $ i_t = i_{m+t}\in \{ i_v, i_{v+1} \} $.
 If
 $ v \neq t $, 
 then
 $ v \neq t \neq v+1 $.
 If
 $ v = t $,
 using that 
 $ \sigma_t = \sigma_{m+t} $
 we get 
 $ i_t = i_{t + 1} $.
 Either way, there is some
 $ w \in [ b-1] \setminus \{ t\} $
 such that 
 $ (i_1, \dots, i_{b}) \in
  A_{\pi, \overrightarrow{\sigma}}( [ b-1]) $
  implies that 
  $ i_w = i_t $,
  hence
  $ \big|A_{\pi, \overrightarrow{\sigma}}( [ b-1]) \big| 
  \leq b-1   $,
  that is 
  $ \mathfrak{a}_{\pi, \overrightarrow{\sigma}} \leq -2 $,
  and the conclusion follows.

\end{proof}

\begin{proof}[Proof of Theorem 5.1]
	${}$
	
Since for each
$ s \leq n $, 
we have that the sequence
$ ( \sigma_{s, N})_N $ 
is either the sequence of identity permutations  or the sequence of matrix transposes, we can simplify the writing by omitting the index $N$ and writing $\sigma_s $ for $\sigma_{s, N}$. With this convention,  it suffices to show that
\begin{align*}
&\lim_{N \rightarrow \infty}  \big[ \mathbb{E} \circ \Tr ( G_N^{\sigma_{1}} \cdots G_N^{ \sigma_{n}} )
- N \lim_{N \rightarrow \infty } \mathbb{E} \circ \tr 
( G_N^{\sigma_{1}} \cdots G_N^{ \sigma_{n}} ) 
\big]\\
&= \sum_{ \rho \in NC(n) }
\big[ \sum_{ \substack{ B \in \rho\\
		B = ( i(1), \dots, i(p) )} }
	\kappa^\prime_p ( \sigma_{i(1)}, \dots, \sigma_{i(p)})
	\cdot \prod_{\substack{D \in \rho \setminus \{ B\} \\
	D =( j(1), \dots, j(r)) }}	\kappa_r (\sigma_{j(1)},\dots,\sigma_{j(r)}) \big]
\end{align*}
where
\begin{align*}
\kappa_r (\sigma_1, \dots, \sigma_r ) = \left\{
\begin{array}{ll}
1 & \textrm{ if } r=2 \textrm{ and } \sigma_1 = \sigma_2 \\
0 & \textrm{ otherwise}
\end{array} 
\right.
\end{align*}
and
\begin{align*}
\kappa^\prime_p (\sigma_1, \dots, \sigma_p ) = \left\{
\begin{array}{ll}
1 & \textrm{ if } p=2m  \textrm{ and } \sigma_{s} \neq \sigma_{m + s } \textrm{ for } s \in [ m ]\\
0 & \textrm{ otherwise.}
\end{array} 
\right.
\end{align*}

On the other hand, from Lemma \ref{lemma:v:pi}, we have that
\begin{align*}
  \mathbb{E} \circ \tr 
( G_N^{\sigma_{1}} \cdots G_N^{ \sigma_{m}} ) 
 = \big| P_2 ( \overrightarrow{\sigma}, 0) \big| 
 + \frac{1}{N} \big| P_2 ( \overrightarrow{\sigma}, 1)\big| + O(N^{-2}) 
\end{align*}
which gives
\begin{align*}
\lim_{N \rightarrow \infty}  \big[ \mathbb{E} \circ \Tr ( G_N^{\sigma_{1}} \cdots G_N^{ \sigma_{m}} )
- N \lim_{N \rightarrow \infty } \mathbb{E} \circ \tr 
( G_N^{\sigma_{1}} \cdots G_N^{ \sigma_{m}} ) 
\big]  &= \big| P_2 ( \overrightarrow{\sigma}, 1) \big|\\
&= \sum_{\rho\in NC(n)}\chi_{P_2 ( \overrightarrow{\sigma}, 1)} (\rho)
\end{align*}
where
 $ \displaystyle \chi_{P_2 ( \overrightarrow{\sigma}, 1)} (\rho) = 
 \left\{
 \begin{array}{ll}
1 &
 \textrm{ if } \rho \in  P_2 (\overrightarrow{\sigma}, 1) 
	\\
	0 &  \textrm{ otherwise.}
	\end{array}
 \right. $
  
  But the definition of 
$  P_2 ( \overrightarrow{\sigma}, 1) $ reads 
\[ 
 \chi_{P_2 ( \overrightarrow{\sigma}, 1)} (\rho) =
 \sum_{ \substack{ B \in \rho\\
 		B = ( i(1), \dots, i(p) )} }
 \kappa^\prime_p ( \sigma_{i(1)}, \dots, \sigma_{i(p)})
 \cdot \prod_{\substack{D \in \rho \setminus \{ B\} \\
 		D =( j(1), \dots, j(r)) }}	\kappa_r (\sigma_{j(1)},\dots,\sigma_{j(r)}) 
\]
and the conclusion follows.

	\end{proof}

\section*{Acknowledgement}
We would like to thank the anonymous referee for careful reading of the manuscript, and for important remarks which led to an improvement of this paper.
\bigskip


\bibliographystyle{abbrv}
\bibliography{main2.bib}

	





\end{document}